\documentclass[11pt]{article}
\usepackage[usenames,dvipsnames,svgnames,table]{xcolor} 

\usepackage{enumitem} 
\usepackage{rotfloat} 
\usepackage{graphicx}
\usepackage{epsfig}
\usepackage{amssymb, amsmath}
\usepackage{verbatim}
\usepackage{natbib}
\usepackage{authblk}
\usepackage{kotex}
\usepackage{multirow}
\usepackage[colorlinks]{hyperref} 
\usepackage[colorinlistoftodos, textsize=scriptsize]{todonotes} 
\usepackage{subfigure} 

\newtheorem{theorem}{Theorem}[section]

\newenvironment{proof}[1][Proof]{\begin{trivlist}
		\item[\hskip \labelsep {\bfseries #1}]}{\end{trivlist}}

\DeclareMathOperator*{\argmax}{argmax}

\newcommand{\bea}{\begin{eqnarray*}}
	\newcommand{\eea}{\end{eqnarray*}}
\newcommand{\bean}{\begin{eqnarray}}
\newcommand{\eean}{\end{eqnarray}}

\newcommand{\bfX}{{\bf X}}

\newcommand{\sg}{\Sigma}
\newcommand{\what}{\widehat}

\newcommand{\lra}{\longrightarrow}

\newcommand{\calD}{\mathcal{D}}

\newcommand{\calP}{\mathcal{P}}

\newcommand{\bbP}{\mathbb{P}} 
\newcommand{\bbR}{\mathbb{R}}
\newcommand{\bbE}{\mathbb{E}}

\begin{document}

\title{Bayesian inference for high-dimensional decomposable graphs}
\author[1]{Kyoungjae Lee}
\affil[1]{Department of Statistics, Inha university}
\author[2]{Xuan Cao}
\affil[2]{Department of Mathematical Sciences, University of Cincinnati}

\maketitle
\begin{abstract}
	In this paper, we consider high-dimensional Gaussian graphical models where the true underlying graph is decomposable.
	A hierarchical $G$-Wishart prior is proposed to conduct a Bayesian inference for the precision matrix and its graph structure.
	Although the posterior asymptotics using the $G$-Wishart prior has received increasing attention in recent years, most of results assume  moderate high-dimensional settings, where the number of variables $p$ is smaller than the sample size $n$. 
	However, this assumption might not hold in many real applications such as genomics, speech recognition and climatology.
	Motivated by this gap, we investigate asymptotic properties of posteriors under the high-dimensional setting where $p$ can be much larger than $n$.
	The pairwise Bayes factor consistency, posterior ratio consistency and graph selection consistency are obtained in this high-dimensional setting.
	Furthermore, the posterior convergence rate for precision matrices under the matrix $\ell_1$-norm is derived, which turns out to coincide with the minimax convergence rate for sparse precision matrices.		
	A simulation study confirms that the proposed Bayesian procedure outperforms  competitors.		
\end{abstract}




\section{Introduction}\label{sec:intro}

Consider a sample of observations from a $p$-dimensional normal model 
\bea
X_1, \ldots, X_n \mid \Omega &\overset{iid}{\sim}& N_p( 0, \Omega^{-1}) ,
\eea 
where $\Omega$ is a $p\times p$ precision matrix.
The main focus of this paper is estimating the (i) support of the precision matrix and (ii) precision matrix itself.
The support recovery of the precision matrix (or equivalently, graph selection) means estimating the locations of nonzero entries of the precision matrix.
A statistical inference on a precision matrix, or a covariance matrix $\sg = \Omega^{-1}$, is essential to uncover the dependence structure of multivariate data.
Especially, a precision matrix reveals the conditional dependences between the variables.
However, especially when the number of variables $p$ can be much larger than the sample size $n$, it is a challenging task because a consistent estimation is impossible without further assumptions \citep{lee2018optimal}.

Various restrictive matrix classes have been suggested to enable consistent estimation in such high-dimensional settings.
One of the most popular restrictive matrix classes is the set of sparse matrices. 
The sparsity assumption, which means most of entries of a matrix are zero, can be imposed on covariance matrices \citep{cai2012optimal, cai2016estimatingb}, precision matrices \citep{cai2016estimating, banerjee2015bayesian} or Cholesky factors \citep{lee2017estimating,lee2019minimax, cao2019posterior}.
In this paper, we focus on sparse precision matrices. 
They lead to sparse Gaussian graphical models, which will be described in Section \ref{subsec:GGM}.
Various statistical methods have been proposed in the frequentist literature for estimating high-dimensional sparse precision matrices using penalized likelihood estimators \citep{yuan2007ggm,rothman2008sparse,ravikumar2011high} and neighborhood-based methods 	\citep{meinshausen2006high,cai2011constrained}.
\cite{ren2015asymptotic} and \cite{cai2016estimating} suggested a regression-based method and  an adaptive constrained $\ell_1$-minimization method, respectively, and showed that the proposed methods achieve the minimax rates and graph selection consistency for sparse precision matrices.

On the Bayesian side, relatively few works have investigated asymptotic properties of posteriors for high-dimensional precision matrices.
The main obstacle is the difficulty of constructing a convenient prior for sparse precision matrices.
Because priors have to be defined on the space of sparse positive definite matrices, calculating normalizing constants is a nontrivial issue.
\cite{banerjee2015bayesian} used a mixture of point mass at zero and Laplace priors for off-diagonal entries and exponential priors for diagonal entries under the positive definiteness constraint.
They obtained the posterior convergence rate for sparse precision matrices under the Frobenius norm, but their result requires the assumption $p = o(n)$.	
Furthermore, because the marginal posterior of the graph is intractable,  they used Laplace approximation.
\cite{wang2015scaling} proposed a similar method by using continuous spike-and-slab priors for off-diagonal entries of precision matrices.
However, theoretical properties of the induced posteriors are unavailable, and a Gibbs sampling algorithm should be used due to the unknown normalizing constant.

As an alternative, the $G$-Wishart prior \citep{atay2005monte}  has been widely used to conduct a  Bayesian inference for sparse precision matrices.
One of advantages of this prior is that the prior density has a closed form if the underlying graph is decomposable, where the definition of a decomposable graph will be given in Section \ref{subsec:GGM}.
Based on the $G$-Wishart prior, \cite{xiang2015high} proved the posterior convergence rate for precision matrices under the matrix $\ell_\infty$-norm when the graph is decomposable.
However, they assumed that the graph is known, which is rarely true in real applications. 
\cite{banerjee2014posterior} also used the $G$-Wishart prior and derived the posterior convergence rate for banded (or bandable) precision matrices, whose entries farther than a certain distance from the diagonal are all zeros (or very small).
Since the underlying graph is always decomposable for banded precision matrices, the posterior can be calculated in a closed form.
However, in \cite{xiang2015high} and \cite{banerjee2014posterior}, the graph selection consistency of posteriors has not been investigated.

Recently, \cite{niu2019bayesian} and \cite{liu2019empirical} investigated asymptotic properties of posteriors using $G$-Wishart priors when the true graph is decomposable and unknown.
\cite{niu2019bayesian} established the posterior ratio consistency as well as the graph selection consistency, when $p$ grows to infinity as $n\to\infty$. 
\cite{liu2019empirical} obtained the posterior convergence rate of precision matrices under the Frobenius norm.
However, these works assumed a {\it moderate} high-dimensional setting, where $p = O( n^\delta)$ for some $0<\delta <1$.
To the best of our knowledge, asymptotic properties of posteriors for decomposable Gaussian graphical models in an {\it ultra} high-dimensional setting, say $p \gg n$, have not been established yet.

In this paper, we consider high-dimensional decomposable Gaussian graphical models.
A hierarchical $G$-Wishart prior is proposed for sparse precision matrices.
We fill the gap in the literature by showing that the proposed Bayesian method achieves the graph selection consistency and the posterior convergence rate in high-dimensional settings, even when $p \gg n$.
Under mild conditions, we first show the pairwise Bayes factor consistency (Theorem \ref{thm:PBF_cons}) and posterior ratio consistency (Theorem \ref{thm:post_ratio}).
Furthermore, the graph selection consistency of posteriors (Theorem \ref{thm:selection_cons}) is shown under slightly stronger conditions.
Based on these results, we also show that our method attains the posterior convergence rate for precision matrices (Theorem \ref{thm:post_conv}) under the matrix $\ell_1$-norm, which is faster than the posterior convergence rates obtained in existing literature.
Furthermore, the consistency of the posterior mean is established (Theorem \ref{thm:cons_BE}).
The practical performance of the proposed method is investigated in simulation studies, which shows that our method outperforms the other frequentist methods.

The rest of paper is organized as follows.
In Section \ref{sec:prel}, we introduce notation, Gaussian graphical models, the hierarchical $G$-Wishart prior and the resulting posterior.
In Section \ref{sec:main}, we establish asymptotic properties of posteriors such as the graph selection consistency and posterior convergence rate.
Simulation studies focusing on both the graph selection and covariance estimation are provided in Section \ref{sec:simul}, and a discussion is given in Section \ref{sec:disc}.
The proofs of the main results are provided in the \hyperref[appn]{Appendix}.

\section{Preliminaries}\label{sec:prel}

\subsection{Notation}\label{subsec:notation}
For any positive sequences $a_n$ and $b_n$, we denote $a_n = o(b_n)$, or equivalently, $a_n \ll b_n$, if $a_n/b_n \lra 0$ as $n\to\infty$, and $a_n = O(b_n)$, or equivalently, $a_n \lesssim b_n$, if there exists a constant $C>0$ such that $a_n/b_n \le C$ for all sufficiently large $n$.
We denote $a_n \asymp b_n$ if there exist positive constants $C_1$ and $C_2$ such that $C_1 \le a_n/b_n \le C_2$.
For any $p\times p$ matrix $A= (A_{ij})$, $P \subset \{1,\ldots,p\}$ and $1\le j\le p$, let $A_{P} = (A_{ij} )_{i,j\in P} \in \bbR^{|P|\times |P|}$ and $A_{P j} = (A_{ij})_{i\in P} \in \bbR^{|P|\times 1}$ be submatrices of $A$.
For any $p\times p$ matrix $A$, we define the matrix $\ell_w$-norm by
\bea
\|A\|_w &=& \sup_{x\in \bbR^p, \|x\|_w=1 } \| A x \|_w 
\eea
for any integer $1\le w \le \infty$, where $\|a\|_w$ is the vector $\ell_w$-norm for any $a \in \bbR^p$.
As special cases, we have
\bean
\|A\|_1 &=& \sup_{x\in \bbR^p, \|x\|_1=1 } \| A x \|_1 \,\, = \,\, \max_{1\le j\le p} \sum_{i=1}^p  | A_{ij} |  , \nonumber\\ 
\|A\| &=& \|A\|_2 \,\,=\,\, \sup_{x\in \bbR^p, \|x\|_2=1 } \| A x \|_2 \label{spectral} \\ 
&=& \big\{ \lambda_{\max}(A^T A)  \big\}^{1/2},  \nonumber
\eean
where $\lambda_{\max}(A)$ is the largest eigenvalue of $A$.
The matrix $\ell_2$-norm, \eqref{spectral}, is called the spectral norm.

\subsection{Gaussian graphical models}\label{subsec:GGM}

Consider an undirected graph by $G= (V,E)$, where $V= \{1,\ldots,p\} = [p]$ and $E \subseteq \{(i,j): i<j , (i,j) \in V \times V \}$.
For simplicity, we denote the number of edges in a graph $G$ by $|G|$.
Let $P_G$ be the set of all $p\times p$ positive definite matrices $\Omega = (\Omega_{ij} )$ with $\Omega_{ij} \neq 0$ if and only if $(i,j) \in E$.
Suppose that we observe the data from the $p$-dimensional Gaussian graphical model,
\bean\label{model}
X_{1}, \ldots, X_{n} \mid \Omega &\overset{iid}{\sim}& N_p (0, \Omega^{-1}) ,
\eean
where $\Omega \in P_G$ is a precision matrix.
Since the graph $G$ is usually unknown, both recovery of the graph $G$ and estimation of the precision matrix $\Omega$ are the main goals of this paper.
We consider the high-dimensional setting where $p=p_n$ grows to infinity as the sample size $n$ gets larger.

We present here some necessary background on graph theory  to be self-contained.
A graph is said to be complete if all vertices are joined by an edge, and a complete subgraph that is maximal is called a clique.
For given vertices $v$ and $w$ in $V$, a path of length $k$ from $v$ to $w$ is a sequence of distinct vertices $v_0, v_1,\ldots, v_k$ such that $v_0=v$, $v_k=w$ and $(v_{i-1}, v_i) \in E$ for all $i=1,\ldots,k$.
As a special case, if $v = w$, then the path is called the cycle of length $k$.
A chord is an edge between two vertices in a cycle but itself is not a part of the cycle.
An undirected graph $G$ is said to be decomposable if every cycle of length greater than or equal to 4 possesses a chord \citep{lauritzen1996graphical}. 
One of the advantages of working with a decomposable graph $G$ is that, for any  decomposable graph $G$, there exist a perfect sequence of cliques $P_1, \ldots, P_h$ and the separators $S_2,\ldots, S_h$ defined as $S_l = ( \cup_{j=1}^{l-1}P_j ) \cap P_l$ for $l=2,\ldots, h$ (\cite{lauritzen1996graphical}, Proposition 2.17).
Here, a sequence is said to be perfect if every $S_l$ is complete and, for all $j>1$, there exists a $l <j$ such that $S_j \subseteq P_l$. 
In this paper, we will focus on decomposable graphs mainly to exploit this property.

\subsection{Hierarchical $G$-Wishart prior}\label{subsec:prior}

We consider a hierarchical prior for the precision matrix $\Omega$ in \eqref{model}.
First, we impose the following prior on the graph $G$, 
\bean\label{prior_G}
\pi(G) &\propto& \binom{p(p-1)/2}{|G|}^{-1}  \exp \big\{ - |G|  \, C_\tau \log p \big\} \, I( G \in  \calD , \,\,  |G| \le R )   ,
\eean
for some constant $C_\tau>0$ and positive integer $R$, where $\calD$ is a set of all decomposable graphs. 
The condition $|G| \le R$ implies that we focus only on the graphs not having too large number of edges.
The prior \eqref{prior_G} consists of two parts: priors for the graph size and the locations of edges.
By using the prior \eqref{prior_G}, the prior mass decreases exponentially with respect to the graph size $|G|$, and given a graph size, the locations of edges are sampled from a uniform distribution.
Similar priors have been commonly used in high-dimensional regression \citep{castillo2015bayesian,yang2016computational,martin2017empirical} and covariance literature \citep{lee2019minimax,liu2019empirical}.

For a given graph $G$, we will work with the $G$-Wishart prior \citep{atay2005monte} 
\bea
\Omega \mid G &\sim& W_{G} ( \nu , A ), 
\eea
whose density function is given by
\bea
\pi(\Omega \mid G ) &=& \frac{1}{I_G(\nu, A) } \det(\Omega)^{(\nu-2)/2} \exp \Big\{ - \frac{1}{2} tr( \Omega A)  \Big\}, \quad \Omega \in P_G,
\eea
where $\nu> 2$, $A$ is a $p\times p$ positive definite matrix and $I_G(\nu, A)$ is the normalizing constant.
The normalizing constant   can be calculated in a closed form if the graph $G$ is decomposable. 
The $G$-Wishart prior is one of the most popular prior distributions for  precision matrices in Gaussian graphical models.
For examples, \cite{banerjee2014posterior,xiang2015high} and \cite{liu2019empirical} used the $G$-Wishart prior in high-dimensional settings.

There are four hyperparameters in the proposed hierarchical $G$-Wishart prior: $C_\tau$, $R$, $\nu$ and $A$. 
To obtain desired asymptotic properties of posterior,  appropriate conditions for hyperparameters will be introduced in Section \ref{sec:main}.

\subsection{Posterior}\label{subsec:posterior}

For Bayesian inference on the graph $G$ and precision matrix $\Omega$, the joint posterior $\pi(\Omega, G \mid\bfX_n)$ should be calculated.
Due to the conjugacy of the $G$-Wishart prior, we have
\bea
\Omega \mid G, \bfX_n &\overset{ind}{\sim}& W_{G}(n + \nu , \,  \bfX_n^T \bfX_n + A), \\
\pi(G \mid \bfX_n )
&\propto& f(\bfX_n \mid G) \pi(G)\\
&\propto&  \frac{I_{G} (n+\nu, \bfX_n^T \bfX_n+ A)}{I_{G} (\nu, A)} \pi(G) ,
\eea
where $\bfX_n = (X_{1}, \ldots, X_{n})^T$ and $f(\bfX_n \mid G)$ is the marginal likelihood
\bea
f( \bfX_n \mid G ) 
&=& \int f(\bfX_n \mid \Omega) \pi(\Omega \mid G) d\Omega  \\
&=& (2\pi)^{- n p/2} \frac{I_{G} (n+\nu, \bfX_n^T \bfX_n+ A)}{I_{G} (\nu, A)} .
\eea
The posterior samples of $(G, \Omega)$ can be obtained from $\pi(G\mid\bfX_n)$ and $\pi(\Omega \mid G, \bfX_n)$ in turn.
Because the marginal posterior $\pi(G\mid\bfX_n)$ is only available up to some unknown normalizing constant, Markov chain Monte Carlo (MCMC) methods such as the Metropolis-Hastings (MH) algorithm should be adopted.

\section{Main results}\label{sec:main}

In this section, we show asymptotic properties of the proposed Bayesian procedure in high-dimensional settings.
Let $G_0 = (V, E_0)$ be the true graph, and $P_{0,1}, \ldots, P_{0, h_0}$ and $S_{0,2},\ldots, S_{0, h_0}$ be the corresponding cliques and separators in a perfect ordering.
Let $\Omega_0 =(\Omega_{0,ij})$ and $\sg_0 = (\Sigma_{0,ij}) = \Omega_0^{-1}$ be the true precision and covariance matrices, respectively.
We assume that the data were generated from the $p$-dimensional Gaussian graphical model with the true precision matrix $\Omega_0\in P_{G_0}$, i.e.,
\bea
X_1,\ldots, X_n &\overset{iid}{\sim}& N_p( 0, \Omega_0^{-1}).
\eea
For given a random vector $Y= (Y_1,\ldots, Y_p)^T \sim N_p(0, \sg_0)$ and an index set $S \subseteq [p]
\setminus \{i,j\}$, we denote $\rho_{ij \mid S }$ as the partial correlation between $Y_{i}$ and $Y_{j}$ given $Y_{S} = (Y_k)_{k \in S}$, i.e., $\rho_{ij \mid S } = \Sigma_{0,ij \mid S}/(\Sigma_{0,ii \mid S} \Sigma_{0,jj \mid S} )^{1/2}$, where $\Sigma_{0,ij \mid S} = \Sigma_{0,ij} -   \sg_{0, i S}\sg_{0,S}^{-1} \sg_{0, Sj}$ for any $i,j \in [p]$.
If $S = \phi$, then $\rho_{ij \mid S }$ reduces to the correlation between $Y_{i}$ and $Y_{j}$, $\rho_{ij} = \Sigma_{0,ij}/(\Sigma_{0,ii} \Sigma_{0,jj} )^{1/2}$.

To obtain desired asymptotic properties of posteriors, we assume the following conditions for the true graph and partial correlations.   \\

\noindent
{\bf(A1)} $|G_0| \le  R$   \\
{\bf(A2)} $\max \{ | \rho_{ij \mid S \setminus \{i,j\} }| : (i,j) \in E_0,  S \subseteq [p],  |S| \le 3 R   \}  \le 1- 1/\sqrt{(n\vee p)}$   \\
{\bf(A3)} $\min \{ \rho_{ij \mid S \setminus \{i,j\} }^2:  (i,j) \in E_0,  S \subseteq [p],  |S| \le 3R \}   \ge C_\beta R^2 \log (n\vee p)/n$ for some constant $C_\beta>0$   \\

Condition (A1) says that the size of the true graph $G_0$ is not too large so that it resides in the prior support.
In fact, the upper bound for $|G_0|$ does not need to be exactly equal to $R$, but just less than $R$.
In the literature, \cite{liu2019empirical} and \cite{niu2019bayesian} also introduced similar conditions to control the number of true edges in $G_0$.
Condition (A2) implies that the $i$th and $j$th variables have an imperfect linear relationship.
It means that there is no set of variables $S$ with $|S|\le 3R$ that makes $i$ and $j$ with $(i,j)\in E_0$ have a perfectly linear relationship when the effects of those variables are removed.
Although $1 - 1/\sqrt{(n\vee p)}$ is used as an upper bound for simplicity, a more general upper bound, $1- 1/(n\vee p)^{c}$ for some constant $c>0$, can be used with a proper change in the lower bound of $C_\beta$ in Theorems \ref{thm:PBF_cons} and \ref{thm:post_ratio}.
Let $\min_{S \subseteq [p], |S|\le p}   \rho_{i j \mid S \setminus \{i,j\} } $ be the {\it minimum partial correlation}, then it is nonzero whenever $(i,j) \in E_0$ in a decomposable graph $G_0$ \citep{nie2017inferring}.  
Condition (A3) gives a lower bound for the nonzero partial correlations $\rho_{i j \mid S \setminus \{i,j\} } $ with $|S| \le 3R$ rather than $|S| \le p$.
Note that the left-hand side of condition (A3) is nonzero whenever the minimum partial correlation is nonzero. 
Thus, this is weaker than a condition on the minimum partial correlation.
In our theory, this condition corresponds to the {\it beta-min condition} in the high-dimensional regression literature, which is essential to obtain selection consistency results \citep{yang2016computational,martin2017empirical,cao2019posterior}.
Note that the above conditions are not easy to verify in practice except for some simple situations.
For example, they are easily satisfied when the number of variables $p$ is fixed.
\\

\noindent
{\bf (P1)} Assume that $\nu$ and $C_\tau$ are fixed constants such that $\nu > 2$ and $C_\tau>0$, respectively. 
Further assume that $R = C_r \{ n/ \log (n\vee p) \}^{\xi/2}$ and $A = g \bfX_n^T \bfX_n$, where $g \asymp (n\vee p)^{-\alpha}$ for some constants $C_r > 0$,  $0\le \xi\le 1$ and $\alpha>0$.\\

Here, ``P'' stands for ``prior''.
Condition (P1) is a sufficient condition for hyperparameters to guarantee the desired asymptotic properties of posteriors. 
Together with condition (A1), $R = C_r \{n/ \log (n\vee p)\}^{\xi/2}$ implies that the number of edges in the true graph $G_0$ is at most of order $ \{n/\log (n\vee p)\}^{\xi /2}$.
By choosing the scale matrix $A=g \bfX_n^T \bfX_n$, our prior can be seen as an inverse of the hyper-inverse Wishart $g$-prior \citep{carvalho2009objective}.
\cite{niu2019bayesian} used a similar prior with $g= n^{-1}$ as suggested by \cite{carvalho2009objective}.
Note that the hyperparameter $g$ serves as a penalty term for adding false edges in graphs, thus we essentially use a stronger penalty than \cite{carvalho2009objective} and \cite{niu2019bayesian} if $\alpha>1$.

\subsection{Graph selection properties of posteriors}\label{subsec:selection}

The first property is consistency of pairwise Bayes factors using $G$-Wishart priors.
Consider the hypothesis testing problem $H_0: G=G_0$ versus $H_1: G=G_1$, for some graph $G_1 \neq G_0$.
If we use priors $\Omega \sim W_{G_0}(\nu, A)$ and $\Omega \sim W_{G_1}(\nu, A)$ under $H_0$ and $H_1$, respectively, we support either $H_0$ or $H_1$ based on the Bayes factor $B_{10}(\bfX_n) := f(\bfX_n \mid G_1) / f(\bfX_n \mid G_0)$.
In general, for a given threshold $C_{th} >0$, we support $H_1$ if $\log B_{10}(\bfX_n)> C_{th}$, and support $H_0$ otherwise.
Theorem \ref{thm:PBF_cons} shows that we can consistently support the true hypothesis $H_0: G=G_0$ based on the pairwise Bayes factor $B_{10}(\bfX_n)$ for any $G_1\neq G_0$.

\begin{theorem}[Pairwise Bayes factor consistency]\label{thm:PBF_cons}
	Assume that conditions (A1)--(A3) and (P1) hold with $C_\beta > 10$ and $\alpha  > 5/2$.
	Then, we have
	\bea
	\frac{f(\bfX_n \mid G)}{f(\bfX_n \mid G_0 )} &\overset{p}{\lra}& 0
	\eea
	as $n\to\infty$, for any decomposable graph $G \neq G_0$  such that $|G|\le R$.
\end{theorem}

\cite{niu2019bayesian} showed the convergence rates of pairwise Bayes factor (BF) (in their Theorem 4.1) on some ``good'' set $\Delta_a$ using $g = n^{-1}$, i.e., $\alpha=1$ in our notation, while we use $\alpha > 5/2$ in Theorem \ref{thm:PBF_cons}.
However, their result neither guarantees  the pairwise BF consistency nor $\mathbb{P}_0(\Delta_a) \to 1$  as $n\to\infty$.
They showed the pairwise BF consistency (in their Corollary 4.1) under the {\it fixed} $p$ setting.
In this setting, the proposed model in this paper also can obtain the pairwise BF consistency using $\alpha = 1$.

The above condition for the hyperparameter $g$, i.e., $g \asymp (n\vee p)^{-\alpha}$ for $\alpha> 5/2$, is an upper bound to obtain the consistency result.
In fact, one can use an exponentially decreasing penalty to prove Theorems \ref{thm:PBF_cons} and \ref{thm:post_ratio} under current conditions, for example, $g \asymp (n \vee p)^{- \tilde{R} \alpha}$ for some $\tilde{R}= \tilde{R}_n \to \infty$ as $n\to\infty$ as long as $\tilde{R} = o(R)$.

For the rest, we consider the hierarchical $G$-Wishart prior described in Section \ref{subsec:prior}.
Theorem \ref{thm:post_ratio} shows what we call as the posterior ratio consistency. 
Note that the consistency of pairwise Bayes factors does not guarantee the posterior ratio consistency, and vice versa.
As a by-product of Theorem \ref{thm:post_ratio}, it can be shown that the posterior mode, $\what{G} = \argmax_{G} \pi(G\mid \bfX_n)$, is a consistent estimator of the true graph $G_0$.

\begin{theorem}[Posterior ratio consistency]\label{thm:post_ratio}
	Assume that conditions (A1)--(A3) and (P1) hold with $C_\beta > 10$ and $\alpha + C_\tau > 3$.
	Then, we have
	\bea
	\frac{\pi(G \mid \bfX_n)}{\pi(G_0\mid \bfX_n )} &\overset{p}{\lra}& 0
	\eea
	as $n\to\infty$, for any decomposable graph $G \neq G_0$.
\end{theorem}

To obtain the posterior ratio consistency, \cite{niu2019bayesian} assumed  $p = O(n^{\alpha_1})$ for some $0<\alpha_1<1/2$, whereas we do not have any condition on the relationship between $n$ and $p$ as long as $p\to\infty$ as $n\to\infty$.
They also assumed $|G_0| = O(n^{\sigma })$, $1-\max_{(i,j)\in E_0} \rho^2_{ij \mid V\setminus \{i,j\}} \asymp n^{-k}$ and $\min_{(i,j)\in E_0} \rho^2_{ij \mid V\setminus\{i,j\}} \asymp n^{-\lambda}$, for some constants $0\le \sigma \le 2\alpha_1, k\ge 0$ and $0\le \lambda < \min(\alpha_1, 1/2-\alpha_1)$, which correspond to conditions (A1), (A2) and (A3) in this paper, respectively.
However, the comparison with our result is not straightforward because  they imposed conditions on $\max_{(i,j)\in E_0} \rho^2_{ij \mid V\setminus\{i,j\}}$ and $\min_{(i,j)\in E_0} \rho^2_{ij \mid V\setminus\{i,j\}}$, whereas we impose conditions on $\max_{(i,j)\in E_0, |S|\le R} \rho^2_{ij \mid S\setminus\{i,j\}}$ and $\min_{(i,j)\in E_0, |S|\le R} \rho^2_{ij \mid S\setminus\{i,j\}}$.

Next we show the strong graph selection consistency, which is much stronger than the posterior ratio consistency.
To prove Theorem \ref{thm:selection_cons}, we require the following conditions instead of conditions (A3) and (P1):  \\

\noindent
{\bf(B3)} $\min \{ \rho_{ij \mid S \setminus \{i,j\} }^2:  (i,j) \in E_0,  S \subseteq [p],  |S| \le 3R \}   \ge C_\beta R^3 \log (n\vee p)/n$ for some constant $C_\beta>0$ 	\\
\noindent{\bf (P2)} 
Assume that $\nu$ and $C_\tau$ are fixed constants such that $ \nu > 2$ and $C_\tau>0$, respectively. 
Further assume that $R = C_r  \{n/ \log (n\vee p) \}^{\xi/3} $ and $g \asymp (n\vee p)^{-R \alpha}$ for some constants $C_r>0$, $0\le\xi \le 1$ and $\alpha>0$.\\

Condition (B3) gives a larger lower bound for the nonzero partial correlations than condition (A3).
Condition (P2) implies that we further restrict the size of the true graph and use stronger penalty for adding false edges. 
Note that if we assume that the size of  the true graph is bounded above by a constant $C_r$, i.e., assuming $\xi=0$ in condition (P2), then condition (B3) is essentially equivalent to (A3) in terms of the rate.

\begin{theorem}[Strong graph selection consistency]\label{thm:selection_cons}
	Assume that conditions (A1), (A2), (B3) and (P2) hold with $C_\beta > 6$ and $\alpha > 3$.
	Then, we have
	\bea
	\pi \big(G = G_0 \mid \bfX_n \big)  &\overset{p}{\lra}& 1
	\eea
	as $n\to\infty$.
\end{theorem}

\cite{niu2019bayesian} also obtained the strong graph selection consistency under slightly stronger conditions than those they used to prove the posterior ratio consistency. 
However,  their result holds only when $p = o( n^{1/3})$, which does not include the ultra high-dimensional setting, $p \gg n$.

In Theorem \ref{thm:selection_cons}, we use stronger penalty $g \asymp (n\vee p )^{-R \alpha}$ compared with Theorems \ref{thm:PBF_cons} and \ref{thm:post_ratio}.
Note that, in Theorems \ref{thm:PBF_cons} and \ref{thm:post_ratio}, we only need to focus on $f(\bfX_n \mid G)$ or $\pi(G \mid \bfX_n)$ for a given graph $G$.
However, to prove Theorem \ref{thm:selection_cons}, we should deal with multiple graphs simultaneously; for example, it is required that $\pi \big(G \subsetneq G_0 \mid \bfX_n \big)$ converges to zero in probability as $n\to\infty$, where we need to control multiple graphs, $\{G: G \subsetneq G_0 \}$, simultaneously.
To this end, a strong penalty $g \asymp (n\vee p )^{-R \alpha}$ is required to prove Theorem \ref{thm:selection_cons} using current techniques.

\subsection{Posterior convergence rate for precision matrices}\label{subsec:post_conv}

In this section, we establish the posterior convergence rate for high-dimensional precision matrices under the matrix $\ell_1$-norm using the proposed hierarchical $G$-Wishart prior.
To obtain the posterior convergence rate, we further assume the following condition:  \\

\noindent
{\bf(B4)} 
There exists a constant $\epsilon_0>0$ such that $\epsilon_0 \le \lambda_{\min}(\Omega_0) \le \lambda_{\max} (\Omega_0) \le \epsilon_0^{-1}$,
where $\lambda_{\min}(\Omega_0)$ is the smallest eigenvalue of $\Omega_0$. \\

Condition (B4) is the well-known bounded eigenvalue condition for $\Omega_0$, and similar conditions can be found in \cite{ren2015asymptotic}, \cite{banerjee2015bayesian} and \cite{liu2019empirical}.
Recently, \cite{liu2019empirical} obtained the posterior convergence rate for precision matrices under the Frobenius norm without the beta-min condition like condition (B3).
However, they assumed a moderate high-dimensional setting, $p + |G_0| = o( n /\log p)$.
Theorem \ref{thm:post_conv} shows the posterior convergence rate of the hierarchical $G$-Wishart prior under the matrix $\ell_1$-norm in high-dimensional settings, including $ p \gg n$.

\begin{theorem}[Posterior convergence rate]\label{thm:post_conv}
	Assume that conditions (A1), (A2), (B3), (B4)  and (P2) hold with $C_\beta > 6$ and $\alpha > 3$.
	Then, if $\log p = o(n)$,
	\bean\label{post_conv_l1}
	\bbE_0 \Big\{ \pi \Big( \| \Omega - \Omega_0 \|_1 \ge  M   \tilde{s}_0^2  \sqrt{\frac{\log (n\vee p)}{n}}  \mid \bfX_n \Big) \Big\} &\lra& 0 
	\eean
	as $n\to\infty$ for some constant $M>0$, where  $\tilde{s}_0 := \max_{1 \le j \le p} \sum_{i=1}^p I( \Omega_{0,ij}  \neq 0) $, and  $\bbE_0$ denotes the expectation corresponding to the model \eqref{model} with $\Omega = \Omega_0$.
\end{theorem}

Using the $G$-Wishart prior, \cite{xiang2015high} obtained a larger posterior convergence rate, $\tilde{s}_0^{5/2} \{ \log(n\vee p)/n \}^{1/2}$, for a precision matrix $\Omega_0 \in \calP_{G_0}$, where $G_0$ is decomposable and known.
\cite{banerjee2014posterior} derived the same posterior convergence rate for banded precision matrices.
It was unclear whether the posterior convergence rate $\tilde{s}_0^{5/2} \{ \log(n\vee p)/n \}^{1/2}$ using the $G$-Wishart prior can be improved or not.
Our result reveals that this rate can be improved even when the true graph $G_0$ is unknown.

When a point estimation of precision matrices is of interest, one might want to use a consistent Bayes estimator.
However, in general, a posterior convergence rate result does not imply the consistency of the Bayes estimator without further conditions.
In the following theorem, we show the conditional posterior mean, $\mathbb{E}^\pi (\Omega \mid \widehat{G}, \mathbf{X}_n)$, is a consistent estimator, and its convergence rate under the matrix $\ell_1$-norm coincides with the posterior convergence rate in Theorem \ref{thm:post_conv}.
Note that the closed form of $\mathbb{E}^\pi (\Omega \mid \hat{G}, \mathbf{X}_n)$ is available because the posterior mode $\widehat{G}$ is decomposable.

\begin{theorem}[Consistency of Bayes estimator]\label{thm:cons_BE}
	Under the same conditions in Theorem \ref{thm:post_conv}, we have
	\bea
	\bbP_0 \bigg(  \, \big\|  \mathbb{E}^\pi (\Omega \mid \widehat{G}, \mathbf{X}_n) - \Omega_0 \big\|_1  \ge M \tilde{s}_0^2 \sqrt{ \frac{\log (n\vee p)}{n}}   \,\, \bigg)   &\lra& 0 
	\eea
	as $n \to\infty$ for some constant $M>0$.
\end{theorem}

\section{Simulation Studies}\label{sec:simul}

\subsection{Simulation I: Illustration of posterior ratio consistency}

In this section, we illustrate the posterior ratio consistency results in Theorem \ref{thm:post_ratio} using a simulation experiment. 
First note that for a complete graph $G$, the explicit expression of the normalizing constant in the $G$-Wishart prior is given by
\bean \label{normalizing_complete}
I_G(\nu, A) = \frac{   2^{(\nu + p -1)p/2}  \pi^{p(p-1)/4}\prod_{i=0}^{p-1}\Gamma\left(\frac{\nu + p -1-i}{2}\right)}{\left\{ \det(A) \right\}^{\frac{\nu + p -1}{2}}}.
\eean
As shown in \cite{roverato2000cholesky} and \cite{banerjee2014posterior}, for any decomposable graph $G$ with the set of cliques $\{C_1, \ldots, C_h\}$ and the set of separators $\{S_2, \ldots, S_h\}$, the following holds:
\bean \label{normalizing_decom}
I_G(\nu, A) = \frac{\prod_{j = 1}^h I_{C_j}\left(\nu, A_{C_j}\right)}{\prod_{j = 2}^h I_{S_j}\left(\nu, A_{S_j}\right)},
\eean
where $A_{C_j}$ denotes the submatrix of $A$ formed by its columns and rows of indexed in $C_j$.
Note that $I_{C_j}(\cdot,\cdot)$ and $I_{S_j}(\cdot, \cdot)$ can be computed using \eqref{normalizing_complete} because $C_j$ and $S_j$ are complete for any decomposable graph $G$. 
Further note that the explicit form of the marginal likelihood is given by
\bea
f( \bfX_n \mid G) 
&=& (2\pi)^{- n p/2} \frac{I_{G} \left(n+\nu, \bfX_n^T \bfX_n+ A\right)}{I_{G} (\nu, A)}.
\eea
It then follows from \eqref{normalizing_decom} that for any decomposable graph $G$, we have
\bean \label{marginal_likelihood}
f( \bfX_n \mid G)  = (2\pi)^{- n p/2} \frac{\prod_{j = 1}^h I_{C_j}\left(n+\nu, (\bfX_n^T \bfX_n+ A)_{C_j}\right)}{\prod_{j = 2}^h I_{S_j}\left(n+\nu, (\bfX_n^T \bfX_n+ A)_{S_j}\right)}
\frac{\prod_{j = 2}^h I_{S_j}\left(\nu, A_{S_j}\right)}{\prod_{j = 1}^h I_{C_j}\left(\nu, A_{C_j}\right)}  . 
\eean
Therefore, we can use \eqref{marginal_likelihood} and prior \eqref{prior_G} to compute the posterior ratio between any two decomposable graphs. 

Next, we consider seven different values of $p$ ranging from $50$ to $350$, and fix $n = 150$. Then, for each fixed $p$, we construct a $p \times p$ covariance matrix $\Sigma_{0,ij} = 0.5^{|i-j|}$ for $1 \le i,j\le p$ such that the inverse covariance matrix $\Omega_0 =\sg_0^{-1}$ will possess a banded structure, i.e., the so-called AR(1) model. The matrix $\Omega_0$ also gives us the structure of the true underlying graph $G_0$. Next, we generate $n$ random samples from  $N_p(0, \Sigma_0)$  to construct our data matrix $\bfX_n$, and set the 
hyperparameters as  $A = 0.1 \delta^{-1} p^{-2.5-\delta}\bfX_n^T\bfX_n$, $\delta = 0.01$, $\nu = 3$ and $C_\tau = 0.5$.
The above process ensures all the assumptions in our Theorem \ref{thm:post_ratio} are satisfied. 
We then examine the posterior ratio under four different cases by computing the log of posterior ratio of a ``non-true" decomposable graph $G$ and $G_0$, $\log \{\pi(G\mid \bfX_n)/\pi(G_0 \mid \bfX_n)\}$, as follows. 
\begin{enumerate}
	\item Case 1: $G$ is a supergraph of $G_0$ and the number of total edges of 
	$G$ is exactly twice of $G_0$, i.e. $|G| = 2 |G_0|$. 
	\item Case 2: $G$ is a subgraph of $G_0$ and the number of total edges of 
	$G$ is exactly half of $G_0$, i.e. $|G| = \frac 1 2 
	|G_0|$. 
	\item Case 3: $G$ is not necessarily a supergraph of $G_0$, but the number 
	of total edges of $G$ is twice of $|G_0|$. 
	\item Case 4: $G$ is not necessarily a subgraph of $G_0$, but the number 
	of total edges of $G$ is half of $|G_0|$. 
\end{enumerate}
\begin{figure}[!tb]
	\centering
	\begin{subfigure}
		{\includegraphics[width=60mm]{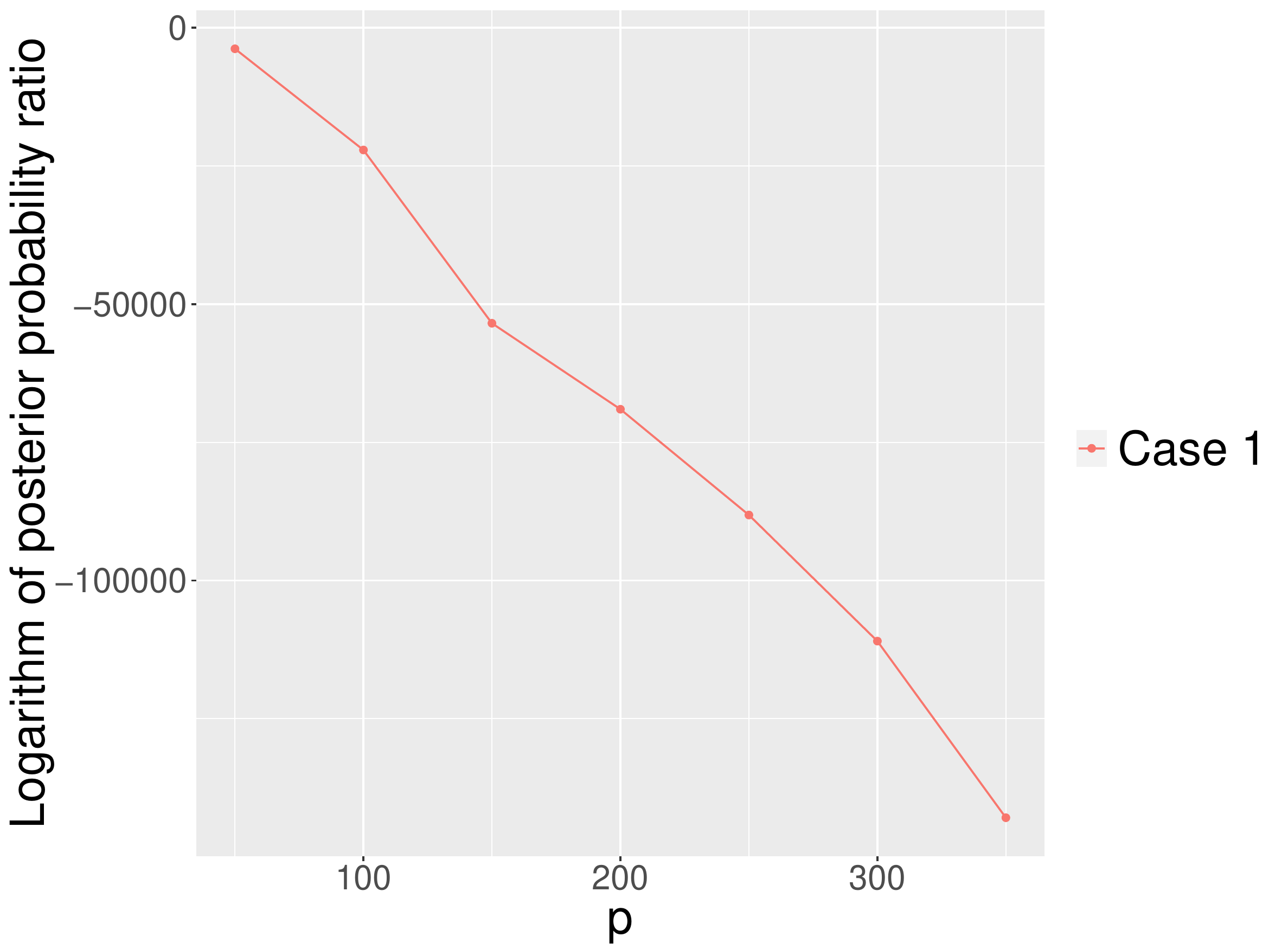}}
	\end{subfigure}
	\qquad
	\begin{subfigure}
		{\includegraphics[width=60mm]{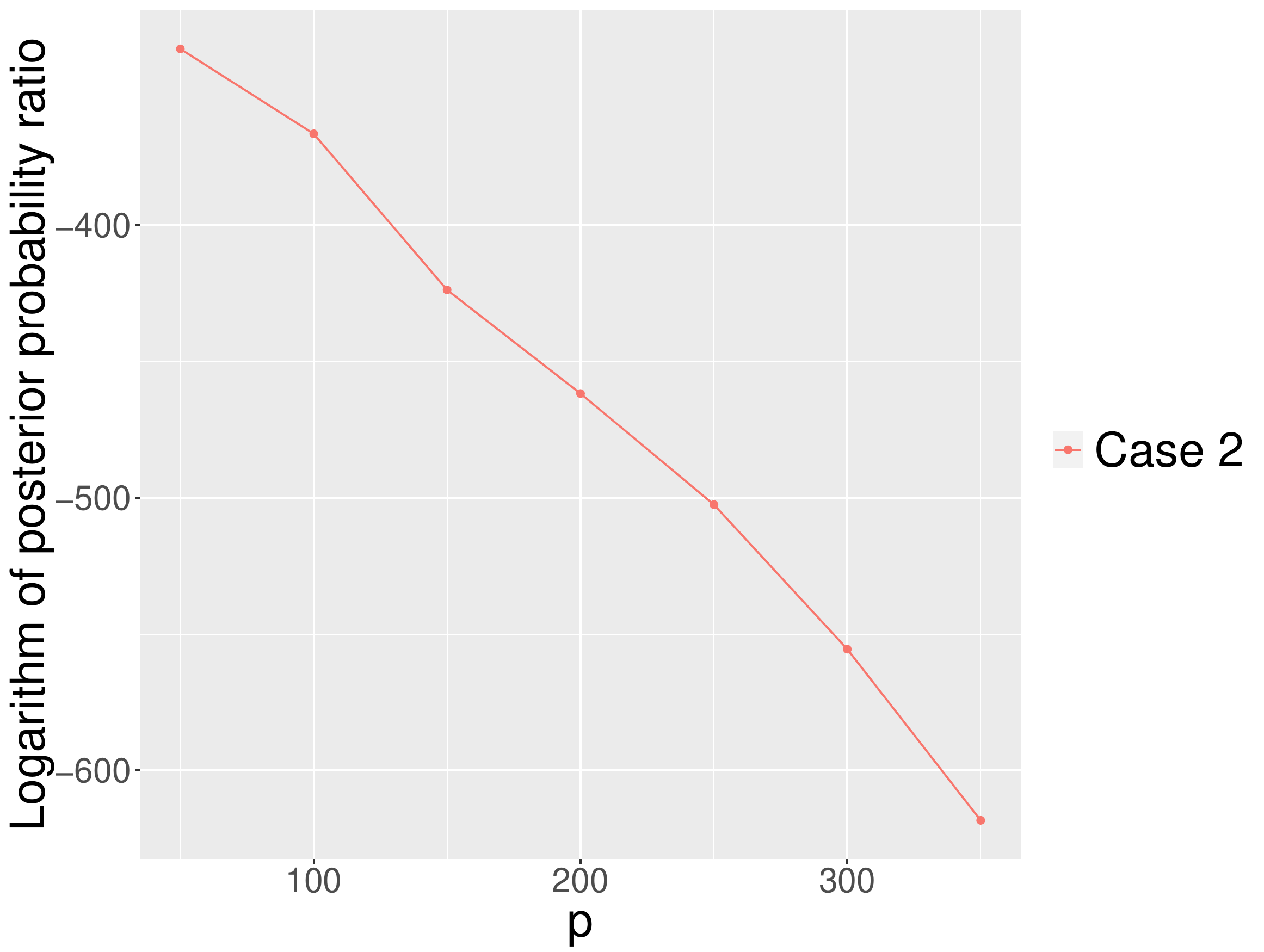}}
	\end{subfigure}

	\begin{subfigure}
		{\includegraphics[width=60mm]{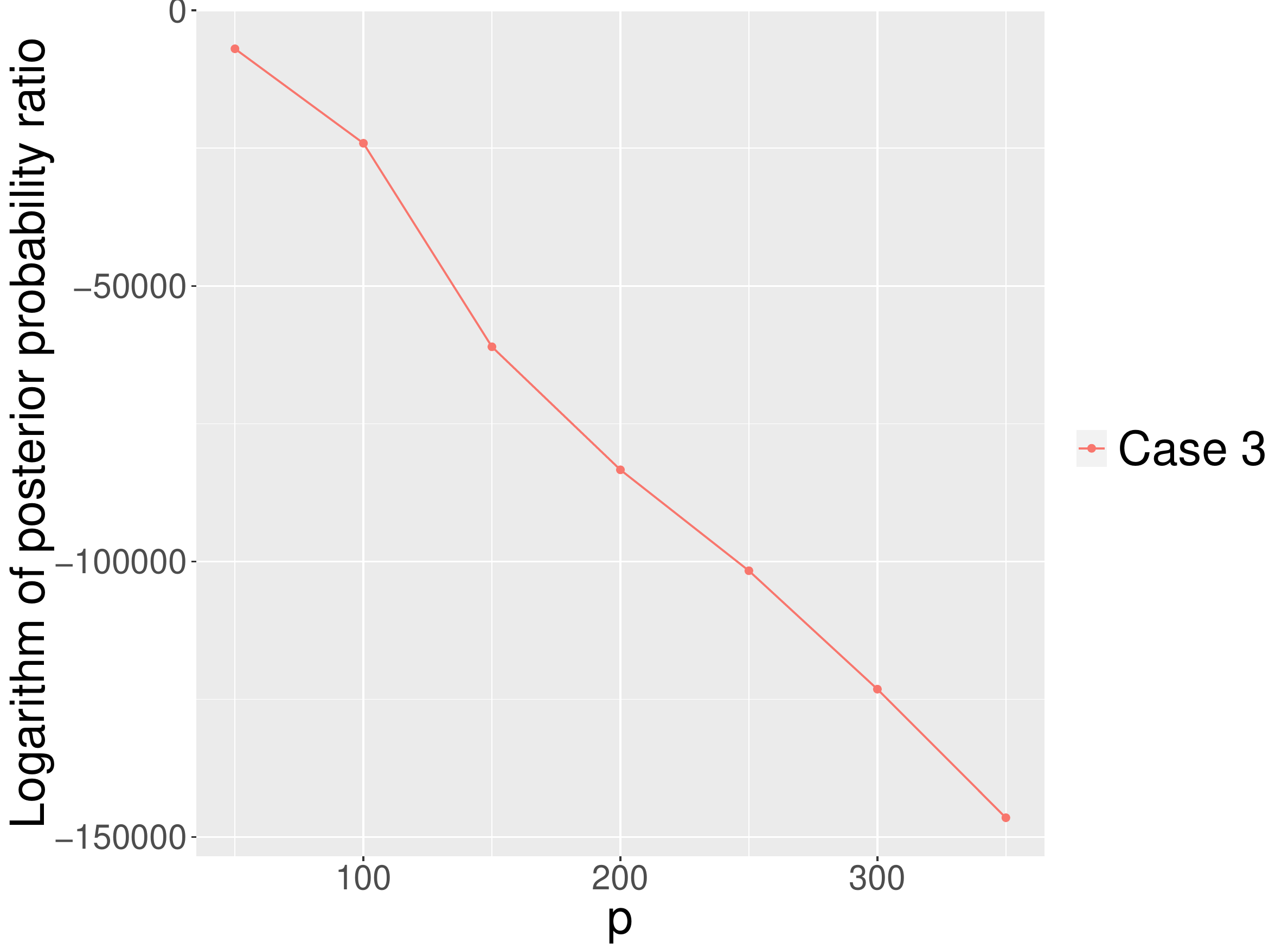}}
	\end{subfigure}
	\qquad
	\begin{subfigure}
		{\includegraphics[width=60mm]{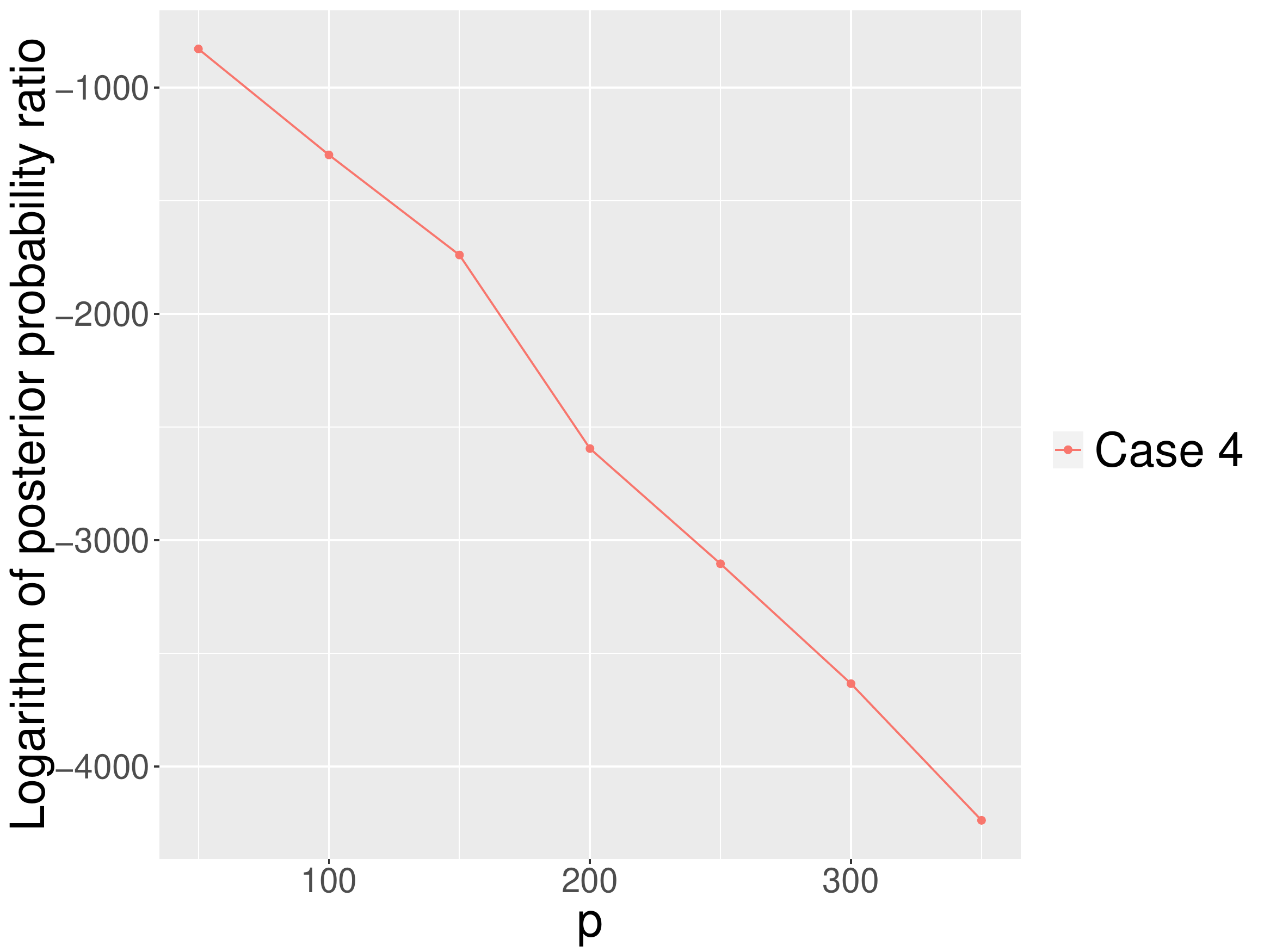}}
	\end{subfigure}

	\caption{Logarithm of posterior probability ratio for $G$ and $G_0$ for various choices of the ``non-true" graph $G$.}
	\label{posterior_ratio_plot}
\end{figure}

\noindent
The logarithms of the posterior ratio for various cases are provided in Figure \ref{posterior_ratio_plot}. 
As expected in all four cases, the logarithm of the posterior ratio decreases as $p$ becomes large. 
Based on the proof of Theorem \ref{thm:post_ratio}, we can see that the posterior ratio $\pi(G\mid \mathbf{X}_n) / \pi (G_0 \mid \mathbf{X}_n)$ converges in probability to zero as $(n\vee p)\to\infty$. 
Thus, this result provides a numerical illustration of Theorem \ref{thm:post_ratio}.

\subsection{Simulation II: Illustration of graph selection} \label{sec:illustration:graph:selection}
In this section, we perform the graph selection procedure under the proposed hierarchical $G$-Wishart prior and evaluate its performance along with other competing methods. Recall that the marginal posterior for $G$ is given by
\bea
\pi(G \mid \bfX_n )
&\propto& f(\bfX_n \mid G) \pi(G)\\
&\propto&  \frac{I_{G} (n+\nu, \bfX_n^T \bfX_n+ A)}{I_{G} (\nu, A)} \pi(G)
\eea
and available up to some unknown normalizing constant. We thereby suggest using the following MH algorithm for posterior inference: 
\begin{enumerate}
	\item Set the initial value $G^{(1)}$. 
	\item For each $s=2,\ldots, S$,  
	\begin{enumerate}
		\item sample $G^{new} \sim q(\cdot \mid G^{(s-1)})$ until $G^{new}$ is decomposable;
		\item set $G^{(s)}=G^{new}$ with the probability 
		\bea
		p_{acc} &=& \min \left\{ 1,  \frac{\pi(G^{new} \mid \bfX_n )}{\pi(G^{(s-1)} \mid \bfX_n )}  \frac{q(G^{(s-1)} \mid G^{new}) }{q(G^{new} \mid G^{(s-1)})}  \right\},
		\eea
		otherwise set $G^{(s)}= G^{(s-1)}$.
	\end{enumerate}
\end{enumerate}
In the above Step 2(a), we verify whether the resulting graph from local perturbations of the current graph is still decomposable by accepting only those moves that satisfy two conditions outlined in \cite{green2013juction} on the junction tree representation of the proposed graph. The proposal kernel $q(\cdot \mid G')$ is chosen such that a new graph $G^{new}$ is sampled by changing a randomly chosen nonzero entry in the lower triangular part of the adjacency matrix for $G'$ to $0$ with probability $0.5$ or by changing a randomly chosen zero entry to $1$ randomly with probability $0.5$. 
We will refer to our proposed method as the MCMC-based graph selection with hierarchical $G$-Wishart distribution (HGW-M).

Following the simulation settings in \cite{yuan2007ggm} and \cite{Friedman2007glasso}, we consider five different structures of the true graph, which corresponds to the following sparsity patterns of the true inverse covariance matrix with all the unit diagonals.
\begin{enumerate} 
	\item Setting 1: AR(1) model with ${\Omega}_{i,i-1} = {\Omega}_{i-1,i} = 0.5$ for $1 \le i \le p-1$.
	\item Setting 2: AR(2) model with ${\Omega}_{i,i-1} = {\Omega}_{i-1,i} = 0.5$ for $1 \le i \le p-1$ and  ${\Omega}_{i,i-2} = {\Omega}_{i-2,i} = 0.25$ for $1 \le i \le p-2$.
	\item Setting 3: AR(4) model with ${\Omega}_{i,i-1} = {\Omega}_{i-1,i} = 0.4$ for $1 \le i \le p-1$, ${\Omega}_{i,i-2} = {\Omega}_{i-2,i} = 0.2$ for $1 \le i \le p-2$, ${\Omega}_{i,i-3} = {\Omega}_{i-3,i} = 0.2$ for $1 \le i \le p-3$, and ${\Omega}_{i,i-4} = {\Omega}_{i-4,i} = 0.1$ for $1 \le i \le p-4$.
	\item Setting 4: Star model where every node connects to the first node, with $\Omega_{1,i} = \Omega_{i,1} = 0.2$ for $2 \le i \le p$, and the remaining entries except the diagonals are set to 0.
	\item Setting 5: Circle model with ${\Omega}_{i,i-1} = {\Omega}_{i-1,i} = 0.5$ for $1 \le i \le p-1$, $\Omega_{1,p} = \Omega_{p,1} = 0.4$, and the remaining entries except the diagonals are set to 0.
\end{enumerate}
For each model, we consider two different values of $p = 100$ or 200, and fix $n = 100$. Next, under each combination of the true precision matrix and the dimension, we generate $n$ observations from $N_p(0, \Sigma_0)$. The hyperparameters for HGW were set at as  $A = (0.1\delta)^{-1} p^{-2.5-\delta}\bfX_n^T\bfX_n$, $\delta = 0.001$, $\nu = 3$ and $C_\tau = 0.5$.
The initial state for $G$ was chosen using the graphical lasso (GLasso) \citep{Friedman2007glasso}.
For posterior inference, we draw $3,000$ posterior samples with a burn-in period of $3,000$ and collect the indices with posterior inclusion probability larger than $0.5$. Therefore, the final estimate using HGW-M can be regarded as the median probability model graph structure. 

To compare the selection performance between the median probability model and the posterior mode, we adopt the hybrid graph selection procedure in \cite{cao2019posterior} to navigate through the massive posterior space. For all the penalized likelihood methods \citep{Friedman2007glasso,cai2011constrained,yuan2007ggm},
a user-specified penalty parameter controls the level of sparsity of the resulting estimator. Varying 
values of the penalty parameter provide a range of possible graphs to choose from. This set 
of graphs is referred to as the solution path. The choice of the penalty parameter is typically 
made by assigning a BIC-like score to each graph on the solution path, and choosing the graph with the highest score \citep{cao2019posterior}.
For the Bayesian approach, the posterior probabilities naturally assign a score for 
all the decomposable graph, but the entire graph space is prohibitively large to search 
in high-dimensional settings. To address this, in the context of Gaussian directed acyclic graphical models, \cite{ben2015high} and \cite{cao2019posterior} develop a 
computationally feasible approach which searches around the graphs on the penalized likelihood 
solution path, and demonstrate that significant improvement in 
accuracy can be obtained by searching beyond the penalized likelihood solution paths using 
posterior probabilities. Adapted to our setting, we first vary the tuning parameter in GLasso on a grid from 0.01 to 1.5. For each fixed parameter, we further threshold the inverse covariance matrix estimated by GLasso on a grid from 0 to 0.5 to get a sequence of $5,000$ graphs, and include them in the candidate set. 
We use the same technique in Section \ref{sec:illustration:graph:selection} to ensure the candidate graphs are decomposable. 
The log posterior probabilities are computed for all candidate graphs, and the one with the highest probability is retained. The shotgun stochastic search is implemented to search around the selected graph and to target the posterior mode $\hat G$ \citep{jones2005experiments}. We refer to this hybrid graph selection approach as HGW-$\hat G$.

The performance of HGW-M and HGW-$\hat G$ will be compared with other existing methods including the GLasso \citep{Friedman2007glasso}, the constrained $\ell_1$-minimization for inverse matrix estimation (CLIME) \citep{cai2011constrained} and the tuning-insensitive approach for optimally estimating Gaussian graphical models (TIGER) \citep{liu2017tiger}. 	
The tuning parameters for GLasso and TIGER were chosen by the criterion of StARS, the stability-based method for choosing the regularization parameter in high dimensional inference for undirected graphs \citep{Liu2010STARS}. The penalty parameter for CLIME was selected by 10-fold cross-validation. For GLasso and TIGER, the final model is determined by collecting the nonzero entries in the estimated precision matrix. Since CLIME could not produce exact zeros in our simulation settings, we constructed the final support by thresholding the absolute values of the estimated precision matrix at 0.025.

To evaluate the performance of variable selection, the precision, sensitivity, specificity and Matthews correlation coefficient (MCC) are reported at Tables \ref{table:comp1} to \ref{table:comp4}, where each simulation setting is repeated for 20 times. 
The criteria are defined as
\bea
\text{Precision} &=&  \frac{\mbox{TP}}{\mbox{TP}+\mbox{FP}}, \\
\text{Sensitivitiy}  &=&     \frac{\mbox{TP}}{\mbox{TP}+\mbox{FN}},   \\
\text{Specificity}  &=&    \frac{\mbox{TN}}{\mbox{TN}+\mbox{FP}},  \\
\text{MCC}  &=&     \frac{\mbox{TP} \times \mbox{TN} - \mbox{FP}\times \mbox{FN}}{\sqrt{(\mbox{TP}+\mbox{FP})(\mbox{TP}+\mbox{FN})(\mbox{TN}+\mbox{FP})(\mbox{TP}+\mbox{FN})}}  ,
\eea
where TP, TN, FP and FN are true positive, true negative, false positive and false negative, respectively. For a clear visualization, in Figure \ref{graph_selection_plot}, we plot the heatmaps for comparing the sparsity structure of the precision matrix estimated by different methods under the AR(1) setting and $p  = 100$.
\begin{figure}[!tb]
	\centering
	\begin{subfigure}
		{\includegraphics[width=27mm]{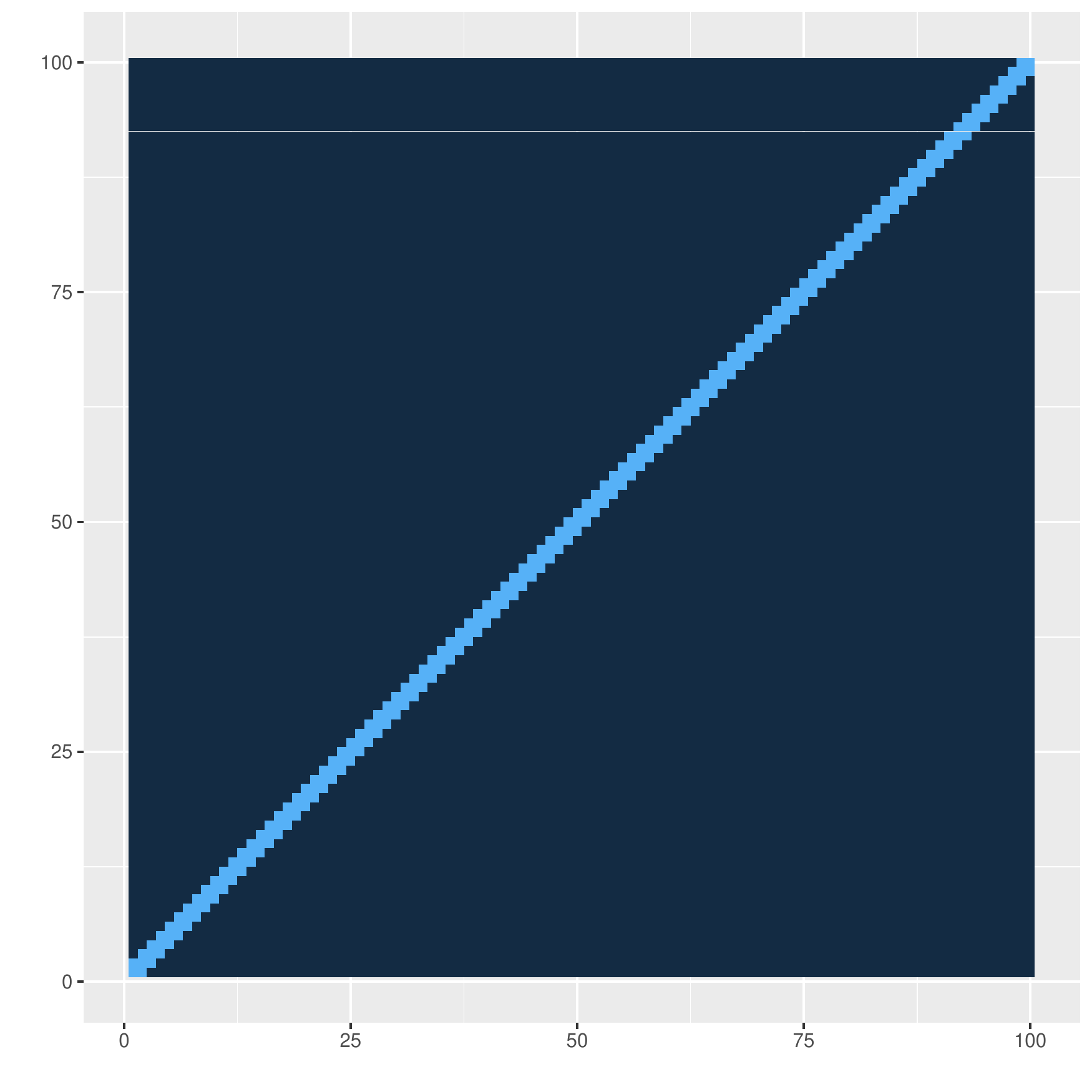}}
	\end{subfigure}
	\begin{subfigure}
		{\includegraphics[width=27mm]{true_graph.pdf}}
	\end{subfigure}	
	\begin{subfigure}
		{\includegraphics[width=27mm]{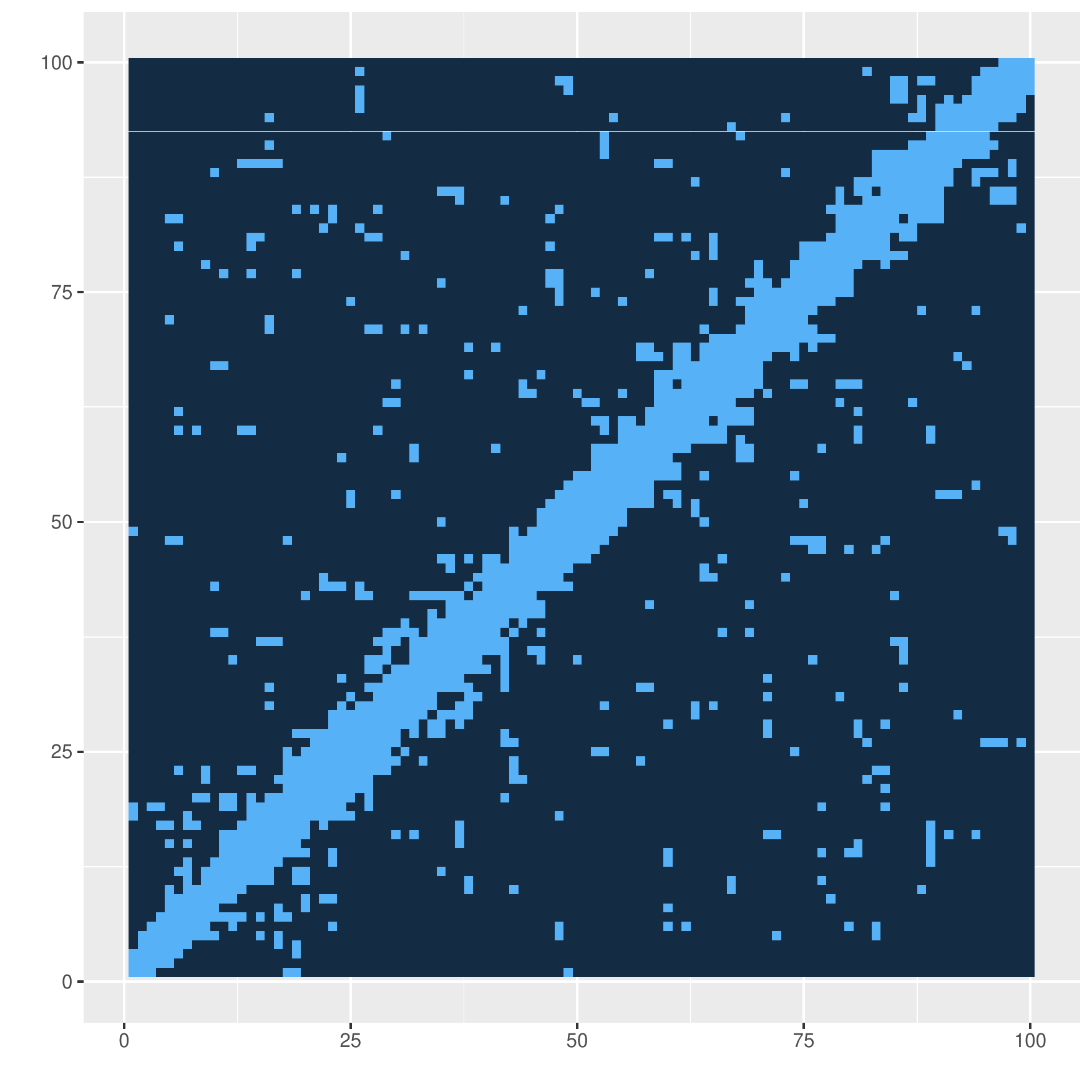}}
	\end{subfigure}
	\begin{subfigure}
		{\includegraphics[width=27mm]{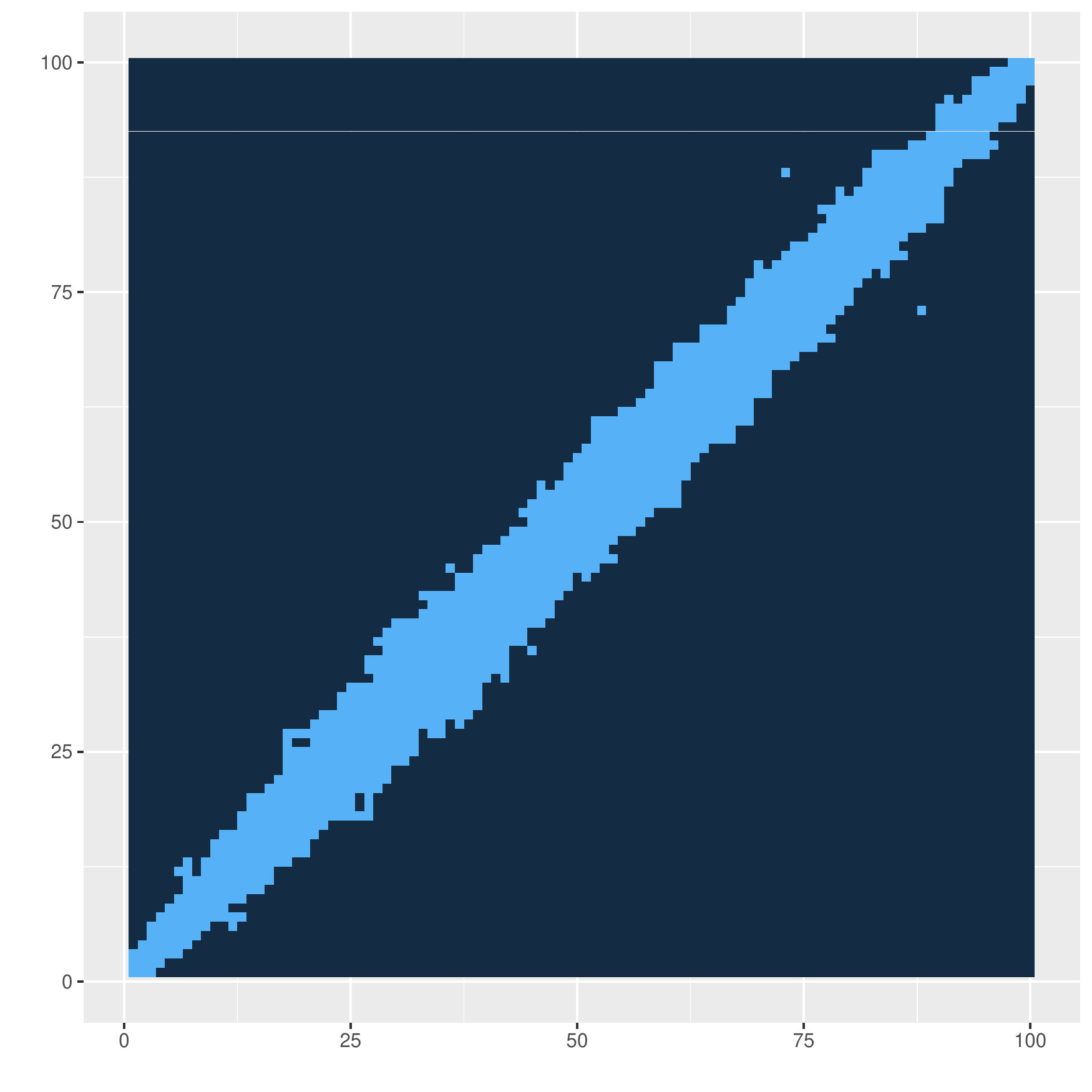}}
	\end{subfigure}
	\begin{subfigure}
		{\includegraphics[width=27mm]{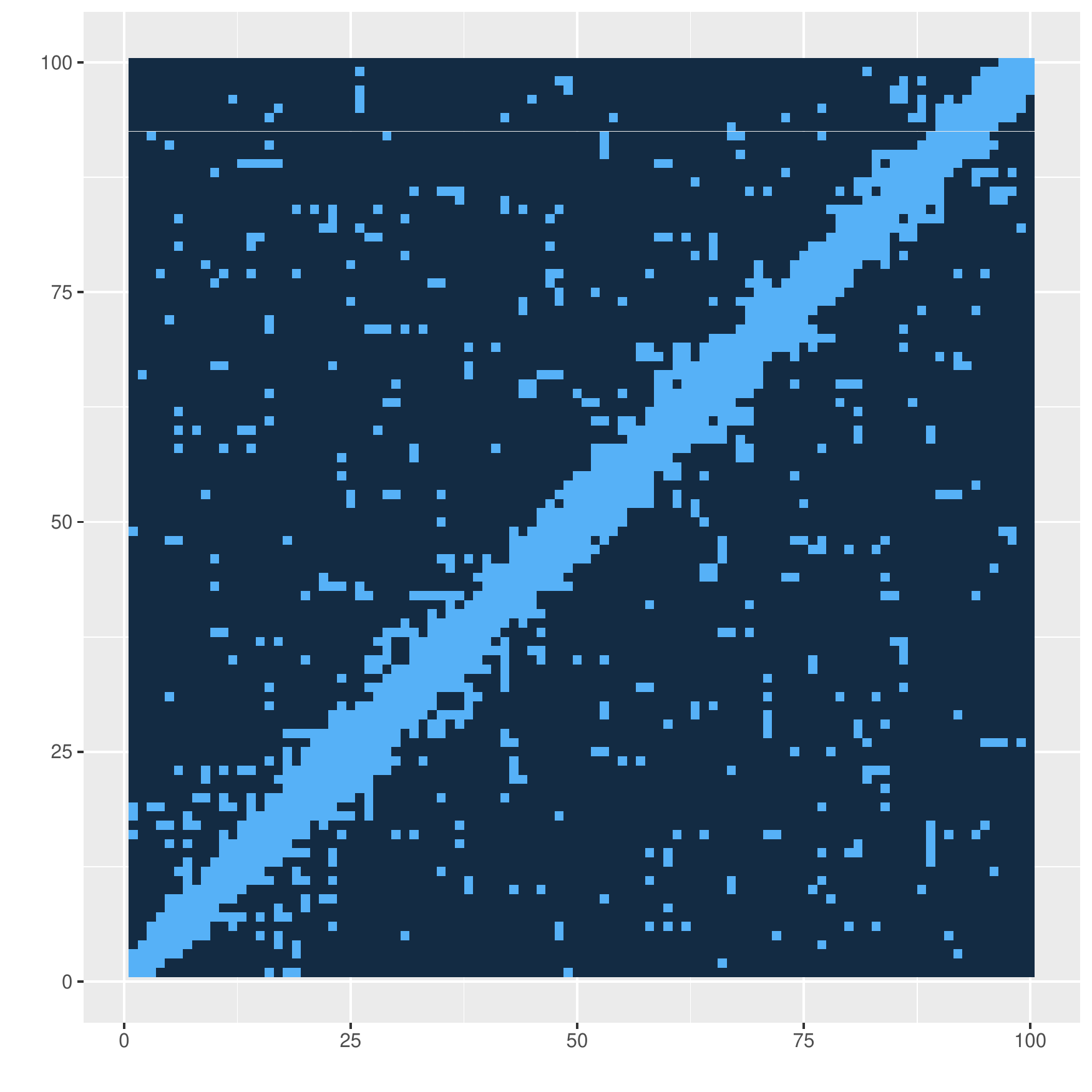}}
	\end{subfigure}
	\caption{Heatmap comparison of the sparsity structure estimated by different methods under the AR(1) setting. Left to right: HGW-M, HGW-$\hat G$, GLasso, CLIME, TIGER.}
	\label{graph_selection_plot}
\end{figure}

\begin{table}[!tb] 
	\centering
	\scalebox{1}{
		\begin{tabular}{ccccccc}
			\hline
			Setting & $p$   & Method & Precision & Sensitivity & Specificity & MCC  \\ \hline
			\multirow{5}{*}{AR(1)} & \multirow{5}{*}{100} & HGW-M  & 1         & 1           & 1           & 1    \\
			&     & HGW-$\hat G$  & 1         & 1           & 1           & 1    \\
			&     & GLasso & 0.15      & 1           & 0.89        & 0.37 \\
			&     & CLIME  & 0.17      & 1           & 0.90         & 0.40  \\
			&     & TIGER  & 0.14      & 1           & 0.88        & 0.36 \\ \hline
			\multirow{5}{*}{AR(1)} & \multirow{5}{*}{200} & HGW-M  & 1         & 1           & 1           & 1    \\
			&     & HGW-$\hat G$  & 0.99      & 1           & 1           & 0.99 \\
			&     & GLasso & 0.12      & 0.99        & 0.93        & 0.33 \\
			&     & CLIME  & 0.14      & 1           & 0.94        & 0.37 \\
			&     & TIGER  & 0.11      & 1           & 0.91        & 0.31\\
			\hline
		\end{tabular}
	}
	\caption{The summary statistics for graph selection under the AR(1) setting with various dimensions are reported for each method.}\label{table:comp1}
\end{table} 

\begin{table}[!tb] 
	\centering
	\scalebox{1}{
		\begin{tabular}{ccccccc}
			\hline
			Setting & $p$   & Method & Precision & Sensitivity & Specificity & MCC  \\ \hline
			\multirow{5}{*}{AR(2)} & \multirow{5}{*}{100} & HGW-M      & 0.94 & 0.57 & 1    & 0.73 \\
			&     & HGW-$\hat G$ & 0.98 & 0.57 & 1    & 0.74 \\
			&     & GLasso   & 0.24 & 0.72 & 0.91 & 0.38 \\
			&     & CLIME    & 0.27 & 0.86 & 0.91 & 0.45 \\
			&     & TIGER    & 0.23 & 0.73 & 0.90  & 0.36 \\ \hline
			\multirow{5}{*}{AR(2)} & \multirow{5}{*}{200} & HGW-M      & 0.91 & 0.49 & 1    & 0.66 \\
			&     & HGW-$\hat G$ & 0.96 & 0.45 & 1    & 0.65 \\
			&     & GLasso   & 0.19 & 0.72 & 0.94 & 0.35 \\
			&     & CLIME    & 0.12 & 0.82 & 0.88 & 0.29 \\
			&     & TIGER    & 0.16 & 0.75 & 0.92 & 0.32 \\
			\hline
		\end{tabular}
	}
	\caption{The summary statistics for graph selection under the AR(2) setting with various dimensions are reported for each method.} \label{table:comp2}
\end{table}

\begin{table}[!tb] 
	\centering
	\scalebox{1}{
		\begin{tabular}{ccccccc}
			\hline
			Setting & $p$   & Method & Precision & Sensitivity & Specificity & MCC  \\ \hline
			\multirow{5}{*}{AR(4)} & \multirow{5}{*}{100} & HGW-M      & 0.96 & 0.13 & 1    & 0.33 \\
			&     & HGW-$\hat G$ & 0.98 & 0.13 & 1    & 0.34 \\
			&     & GLasso   & 0.32 & 0.29 & 0.95 & 0.25 \\
			&     & CLIME    & 0.25 & 0.40  & 0.89 & 0.24 \\
			&     & TIGER    & 0.27 & 0.31 & 0.93 & 0.23 \\ \hline
			\multirow{5}{*}{AR(4)} & \multirow{5}{*}{200} & HGW-M      & 0.81 & 0.12 & 1    & 0.30  \\
			&     & HGW-$\hat G$ & 0.85 & 0.12 & 1    & 0.31 \\
			&     & GLasso   & 0.21 & 0.27 & 0.96 & 0.21 \\
			&     & CLIME    & 0.11 & 0.35 & 0.88 & 0.13 \\
			&     & TIGER    & 0.19 & 0.29 & 0.95 & 0.19 \\
			\hline
		\end{tabular}
	}
	\caption{The summary statistics for graph selection under the AR(4) setting with various dimensions are reported for each method.} \label{table:comp3}
\end{table}

\begin{table}[!tb] 
	\centering
	\scalebox{1}{
		\begin{tabular}{ccccccc}
			\hline
			Setting & $p$   & Method & Precision & Sensitivity & Specificity & MCC  \\ \hline
			\multirow{5}{*}{Star}   & \multirow{5}{*}{100}    & HGW-M      & 1    & 1    & 1    & 1    \\
			&     & HGW-$\hat G$ & 0.99 & 1    & 1    & 0.99 \\
			&     & GLasso   & 0.38 & 1    & 0.97 & 0.61 \\
			&     & CLIME    & 0.13 & 0.79 & 0.90  & 0.30  \\
			&     & TIGER    & 0.33 & 1    & 0.96 & 0.56 \\ \hline
			\multirow{5}{*}{Circle} & \multirow{5}{*}{200} & HGW-M      & 1    & 1    & 1    & 0.99 \\
			&     & HGW-$\hat G$ & 0.99 & 0.99 & 1    & 0.99 \\
			&     & GLasso   & 0.31 & 1    & 0.98 & 0.55 \\
			&     & CLIME    & 0.08 & 1    & 0.87 & 0.26 \\
			&     & TIGER    & 0.28 & 1    & 0.97 & 0.52 \\
			\hline
		\end{tabular}
	}
	\caption{The summary statistics for graph selection under Setting 4 and Setting 5 with various dimensions are reported for each method.
	} \label{table:comp4}
\end{table}
Based on the simulation results, we notice that our methods overall work better than the regularization methods across various settings. 
Our methods perform particularly well in the sparse models under the AR(1), Star and Circle settings. This is because the consistency conditions of HGW are easier to satisfy under sparse settings. Note that when the posterior probability is larger than $1/2$, the median probability model based on HGW-M coincides with the posterior mode based on HGW-$\hat G$ \citep{barbieri2004optimal}.
Because we have proved the strong selection consistency (Theorem \ref{thm:selection_cons}), the two models should be asymptotically equivalent. This is indeed reflected in our simulations, as we notice HGW-M and HGW-$\hat G$ perform comparably well in most settings. Generally speaking, the proposed methods are able to achieve better specificity and precision, while the regularization methods have better sensitivity. The poor specificity of the regularization methods is in accordance with previous work demonstrating that selection of the regularization parameter using cross-validation is optimal with respect to prediction but tends to include more noise predictors compared with Bayesian methods \citep{meinshausen2006high}. 
Overall, our simulation studies indicate that the proposed method can perform well under a variety of configurations with different dimensions, sparsity levels and correlation structures.

\subsection{Simulation III: Illustration of inverse covariance estimation} \label{sec:illustration:estimation}
In this section, we provide the performance comparison for the inverse covariance estimation using different methods. For each fixed $p$, the true inverse covariance matrix
and the subsequent dataset, are generated by the same mechanism as in Section \ref{sec:illustration:graph:selection}. To use HGW-M for the estimation, within each iteration, we sample $\Omega^{(s)} \sim W_{G^{(s)}}(n + \nu , \,  \bfX_n^T \bfX_n + A)$ after Step 2(b), and construct our final estimate by taking the average of all the $\Omega^{(s)}$ after a burn-in period. 
In terms of the Bayes estimators based on the posterior mode, since the posterior mode is also decomposable, the Bayes estimators can be explicitly derived for that graph under various loss functions \citep{bala2008flexible,banerjee2014posterior}. Given the posterior mode, we consider two Bayes estimators $\hat \Omega^{\ell_1}$ and $\hat \Omega^{\ell_2}$ corresponding to the $\ell_1$ Stein's loss and $\ell_2$ squared-error loss, respectively. 
The estimated inverse covariance matrices based on other frequentist approaches are obtained as specified in Section \ref{sec:illustration:graph:selection}.
To evaluate the performance of covariance estimation, different criteria for measuring the estimation loss are reported at Tables \ref{table:comp5} to \ref{table:comp8}, where each simulation setting is repeated for 20 times. 
Relative errors are chosen as criteria. Specifically, for a matrix norm $\|\cdot\|$ and an estimator $\hat{\Omega}$, the relative error is defined as $\|\Omega_0 - \hat{\Omega}\| / \|\Omega_0\|$.
In Tables \ref{table:comp5}--\ref{table:comp8}, $E_1$, $E_2$, $E_3$ and $E_4$ represent the relative errors based on the matrix $\ell_1$-norm, the matrix $\ell_2$-norm (spectral norm), the vector $\ell_2$-norm (Frobenius norm) and the vector $\ell_\infty$-norm (entrywise maximum norm), respectively.
\begin{table}[!tb] 
	\centering
	\scalebox{1}{
		\begin{tabular}{cccccccc}
			\hline
			Setting                & $p$                    & Method               & $E_1$ & $E_2$ & $E_3$ & $E_4$ \\ \hline
			\multirow{6}{*}{AR(1)} & \multirow{6}{*}{100} & HGW-M                & 0.29        & 0.26         & 0.11  & 0.28        \\
			&                      & HGW-$\hat \Omega^{\ell_1}$ & 0.27        & 0.24         & 0.11  & 0.29        \\
			&                      & HGW-$\hat \Omega^{\ell_2}$ & 0.32        & 0.28         & 0.11  & 0.31        \\
			&                      & GLasso               & 0.92        & 0.88         & 0.85  & 0.83        \\
			&                      & CLIME                & 1.02        & 0.60          & 0.48  & 0.37        \\
			&                      & TIGER                & 0.86        & 0.80          & 0.76  & 0.75        \\ \hline
			\multirow{6}{*}{AR(1)} & \multirow{6}{*}{200} & HGW-M                & 0.37        & 0.34         & 0.15  & 0.38        \\
			&                      & HGW-$\hat \Omega^{\ell_1}$ & 0.32        & 0.27         & 0.13  & 0.35        \\
			&                      & HGW-$\hat \Omega^{\ell_2}$ & 0.35        & 0.31         & 0.15  & 0.40         \\
			&                      & GLasso               & 0.93        & 0.87         & 0.84  & 0.83        \\
			&                      & CLIME                & 1.16        & 0.64         & 0.55  & 0.43        \\
			&                      & TIGER                & 0.88        & 0.81         & 0.76  & 0.75       \\ \hline
		\end{tabular}
	}
	\caption{The summary statistics for inverse covariance estimation under the AR(1) setting with various dimensions are reported for each method.}\label{table:comp5}
\end{table} 

\begin{table}[!tb] 
	\centering
	\scalebox{1}{
		\begin{tabular}{cccccccc}
			\hline
			Setting                & $p$                    & Method               & $E_1$ & $E_2$ & $E_3$ & $E_4$ \\ \hline
			\multirow{6}{*}{AR(2)} & \multirow{6}{*}{100} & HGW-M                & 0.68 & 0.54 & 0.41 & 0.49 \\
			&                      & HGW-$\hat \Omega^{\ell_1}$  & 0.80  & 0.57 & 0.39 & 0.50  \\
			&                      & HGW-$\hat \Omega^{\ell_2}$  & 0.78 & 0.56 & 0.39 & 0.50  \\
			&                      & GLasso               & 0.85 & 0.73 & 0.64 & 0.55 \\
			&                      & CLIME                & 1.23 & 0.65 & 0.58 & 1.15 \\
			&                      & TIGER                & 0.85 & 0.72 & 0.63 & 0.55 \\ \hline
			\multirow{6}{*}{AR(2)} & \multirow{6}{*}{200} & HGW-M                & 0.81 & 0.64 & 0.47 & 0.53 \\
			&                      & HGW-$\hat \Omega^{\ell_1}$  & 0.80  & 0.61 & 0.48 & 0.52 \\
			&                      & HGW-$\hat \Omega^{\ell_2}$  & 0.80  & 0.60  & 0.47 & 0.58 \\
			&                      & GLasso               & 0.91 & 0.74 & 0.66 & 0.58 \\
			&                      & CLIME                & 3.48 & 1.85 & 1.24 & 3.95 \\
			&                      & TIGER                & 0.93 & 0.73 & 0.64 & 0.56 \\ \hline
		\end{tabular}
	}
	\caption{The summary statistics for inverse covariance estimation under the AR(2) setting with various dimensions are reported for each method.}\label{table:comp6}
\end{table} 

\begin{table}[!tb] 
	\centering
	\scalebox{1}{
		\begin{tabular}{cccccccc}
			\hline
			Setting                & $p$                    & Method              & $E_1$ & $E_2$ & $E_3$ & $E_4$ \\ \hline
			\multirow{6}{*}{AR(4)} & \multirow{6}{*}{100} & HGW-M                & 0.85 & 0.68 & 0.53 & 0.44 \\
			&                      & HGW-$\hat \Omega^{\ell_1}$ & 0.85 & 0.65 & 0.52 & 0.43 \\
			&                      & HGW-$\hat \Omega^{\ell_2}$ & 0.84 & 0.64 & 0.51 & 0.42 \\
			&                      & GLasso               & 0.82 & 0.72 & 0.6  & 0.49 \\
			&                      & CLIME                & 1.17 & 0.48 & 0.54 & 0.90  \\
			&                      & TIGER                & 0.84 & 0.71 & 0.58 & 0.48 \\ \hline
			\multirow{6}{*}{AR(4)} & \multirow{6}{*}{200} & HGW-M                & 0.94 & 0.67 & 0.53 & 0.59 \\
			&                      & HGW-$\hat \Omega^{\ell_1}$ & 0.87 & 0.72 & 0.57 & 0.53 \\
			&                      & HGW-$\hat \Omega^{\ell_2}$ & 0.87 & 0.71 & 0.56 & 0.52 \\
			&                      & GLasso               & 0.88 & 0.74 & 0.61 & 0.50  \\
			&                      & CLIME                & 2.66 & 1.18 & 1.01 & 2.54 \\
			&                      & TIGER                & 0.91 & 0.73 & 0.60  & 0.48      \\ \hline
		\end{tabular}
	}
	\caption{The summary statistics for inverse covariance estimation under the AR(4) setting with various dimensions are reported for each method.}\label{table:comp7}
\end{table} 

\begin{table}[!tb] 
	\centering
	\scalebox{1}{
		\begin{tabular}{cccccccc}
			\hline
			Setting                & $p$                    & Method              & $E_1$ & $E_2$ & $E_3$ & $E_4$ \\ \hline
			\multirow{6}{*}{Star}   & \multirow{6}{*}{100} & HGW-M                & 0.13 & 0.19 & 0.14 & 0.36 \\
			&                      & HGW-$\hat \Omega^{\ell_1}$ & 0.13 & 0.20  & 0.15 & 0.36 \\
			&                      & HGW-$\hat \Omega^{\ell_2}$ & 0.14 & 0.21 & 0.15 & 0.39 \\
			&                      & GLasso               & 0.27 & 0.29 & 0.21 & 0.39 \\
			&                      & CLIME                & 0.83 & 0.50  & 0.21 & 0.42 \\
			&                      & TIGER                & 0.27 & 0.30  & 0.21 & 0.38 \\ \hline
			\multirow{6}{*}{Circle} & \multirow{6}{*}{200} & HGW-M                & 0.56 & 0.44 & 0.17 & 0.50  \\
			&                      & HGW-$\hat \Omega^{\ell_1}$ & 0.56 & 0.51 & 0.20  & 0.52 \\
			&                      & HGW-$\hat \Omega^{\ell_2}$ & 0.54 & 0.50  & 0.18 & 0.50  \\
			&                      & GLasso               & 0.83 & 0.76 & 0.71 & 0.66 \\
			&                      & CLIME                & 1.14 & 0.63 & 0.54 & 0.40  \\
			&                      & TIGER                & 0.80  & 0.67 & 0.61 & 0.62       \\ \hline
		\end{tabular}
	}
	\caption{The summary statistics for inverse covariance estimation under Setting 4 and Setting 5 with various dimensions are reported for each method.}\label{table:comp8}
\end{table}

In terms of estimating the inverse covariance matrix, we can tell from the simulation results that our methods overall work better than the regularization methods across various settings. 
Similar to the performance for uncovering the true sparsity pattern in Section \ref{sec:illustration:graph:selection}, our methods can more accurately estimate the magnitudes of the true precision matrix in the sparse models under the AR(1), Star and Circle settings. Different Bayes estimators including the MCMC-based estimator, $\hat \Omega^{\ell_1}$ and $\hat \Omega^{\ell_2}$ perform comparably well, which again shows the validity of our theoretical results. 
Overall, our simulation studies indicate that the proposed method can accommodate a variety of configurations with different dimensions and correlation structures for estimating the inverse covariance matrix.

\section{Discussion}\label{sec:disc}
In this paper, we assume that the true graph $G_0$ is decomposable. 
Recently, \cite{niu2019bayesian} showed that, even when $G_0$ is non-decomposable, the marginal posterior of the graph $G$ concentrates on the space of the minimal triangulation of $G_0$.
Here, a triangulation of a graph $G=(V, E)$ is a decomposable graph $G^\Delta = (V, E \cup F)$, where $F$ is called a set of fill-in edges, and a triangulation is minimal if any only if the removal of any single edge in $F$ leads to a non-decomposable graph.
It would be interesting to investigate whether similar properties hold in our setting using the hierarchical $G$-Wishart prior.

Another open problem is whether we can relax the decomposability condition.
We assume that the support of the prior is a subset of all decomposable graphs mainly due to technical reasons.
By focusing on decomposable graphs, the normalizing constants of posteriors are available in closed forms.
This allows us to calculate upper and lower bounds of a posterior ratio.
It is unclear to us whether this decomposability condition can be removed.
Without this condition, general techniques for obtaining posterior convergence rate, for example, Theorem 8.9 in \cite{ghosal2017fundamentals}, might be needed.
\cite{banerjee2015bayesian} used this technique to prove the posterior convergence rate for sparse precision matrices under the Frobenius norm.
However, it might be difficult to obtain the posterior convergence rate under the matrix $\ell_1$-norm using similar arguments in \cite{banerjee2015bayesian}.
Let $\epsilon_n$ and $\tilde{\epsilon}_n$ be the posterior convergence rates for precision matrices under the matrix $\ell_1$-norm and Frobenius norm, respectively, where $\epsilon_n \ll \tilde{\epsilon}_n$.
Then, one can see that it is much more difficult to prove the prior thickness (condition (i) of Theorem 8.9 in \cite{ghosal2017fundamentals}) using $\epsilon_n$.
Therefore, we suspect that the arguments in \cite{banerjee2015bayesian} cannot be directly applied to our setting.

\begin{appendix}
\section{Proofs of main theorems}\label{sec:proofs}

\begin{proof}[Proof of Theorem \ref{thm:PBF_cons}]
	If $G \neq G_0$, then $G_0 \subsetneq G$ or $G_0 \nsubseteq G$.
	We first focus on the case $G_0 \subsetneq G$.
	By Lemma 2.22 in \cite{lauritzen1996graphical}, there exist a sequence of decomposable graphs $G_0 \subset G_1 \subset \cdots \subset G_{k-1} \subset G_k = G$ with $k = |G|-|G_0|$, where $G_0, G_1,\ldots, G_{k-1},G_k$ differ from by exactly one edge.
	Then,
	\bea
	\frac{f(\bfX_n \mid G)}{f(\bfX_n\mid G_0)}
	&=& \frac{f(\bfX_n\mid G_1 )}{f(\bfX_n\mid G_0 )} \frac{f(\bfX_n\mid G_2)}{f(\bfX_n\mid G_1 )} \times \cdots \times \frac{f(\bfX_n\mid G_k)}{f(\bfX_n\mid G_{k-1} )} .
	\eea
	For a given constant $C_1>0$, let $N_l(C_1) = \{ \bfX_n: | \hat{\rho}_{i_l j_l \mid S_l } - \rho_{i_l j_l \mid S_l }  |^2  > C_1 \log (n\vee p)/n  \}$, where $(i_l, j_l)$ is the added edge in the move from $G_{l-1}$ to $G_l$, and $S_l$ is the separator which separates two cliques including $i_l$ and $j_l$ in $G_{l-1}$.
	Note that $\rho_{i_l j_l \mid S_l } = 0$ for any $l=1,\ldots, k$ by Lemma D.4 in \cite{niu2019bayesian}.
	Thus, by the proof of Theorem A.3 and Corollary A.1 in \cite{niu2019bayesian}, we have
	\bea
	&& \bbP_0 ( \cup_{l=1}^k N_l(C_1) )  \\
	&\le& \sum_{l=1}^k \bbP_0 (N_l(C_1) ) \\
	&\le& \sum_{l=1}^k 21 \exp \Big\{ - (n- R) \frac{C_1 \log (n\vee p)}{2n}  \Big\} \Big(  \frac{n}{C_1(n- R)\log (n\vee p)} \Big)^{1/2}  \\
	&\le& 21 (|G| - |G_0| ) \exp \Big\{ - (n- R) \frac{C_1 \log (n\vee p)}{2n}  \Big\} \Big(  \frac{n}{C_1(n- R)\log (n\vee p)} \Big)^{1/2}  \\
	&\le& 21 \exp \Big[   - \Big\{ \frac{C_1}{2}\big( 1- \frac{ C_{r}}{\log (n\vee p)}\big) - 2 \Big\} \log (n\vee p) \Big] ,
	\eea
	which is of order $o(1)$ for any constant $C_1 > 4 + \epsilon'$ and any sufficiently small constant $\epsilon'>0$. 
	Therefore, we can restrict ourselves to event $\cap_{l=1}^k N_l(C_1)^c$.
	Because $\nu >2$ and $\alpha  > 5/2$,
	\bea
	\frac{f(\bfX_n\mid G_l)}{f(\bfX_n\mid G_{l-1} )} 
	&\le& g \Big(\frac{\nu+n +|S_l|}{\nu+|S_l|-1/2} \Big)^{1/2} \big( 1- \hat{\rho}_{i_l j_l \mid S_l}^2 \big)^{-n/2}   \\
	&\lesssim& (n\vee p)^{-\alpha}   \Big( 1+ \frac{n + 1/2}{\nu + |S_l| - 1/2 }\Big)^{1/2} \Big( 1 - \frac{C_1 \log (n\vee p)}{n} \Big)^{-n/2}  \\
	&\le& (n\vee p)^{-\alpha}   n^{1/2} \Big( 1 - \frac{C_1 \log (n\vee p)}{n} \Big)^{-n/2}  \\
	&\le&  \exp \Big\{  - \Big( \alpha -\frac{1}{2} - \frac{C_1}{2} \Big) \log (n\vee p)  \Big\}  
	\eea
	on $\cap_{l=1}^k N_l(C_1)^c$,
	where the first inequality follows from Lemma C.1 in \cite{niu2019bayesian}.
	The last expression is of order $o(1)$ by choosing a constant $C_1$ arbitrarily close to $4$.
	Thus, we have
	\bea
	\frac{f(\bfX_n\mid G )}{f(\bfX_n\mid G_0)} &\overset{p}{\lra}&  0
	\eea
	for any $G_0 \subsetneq G$, as $n\to\infty$.

	Now we consider the case $G_0 \nsubseteq G$.
	Let $(G \cup G_0)_m$ be a minimum triangulation of $G\cup G_0$.
	Note that
	\bea
	\frac{f(\bfX_n\mid G )}{f(\bfX_n\mid G_0)} 
	&=& \frac{f(\bfX_n\mid (G \cup G_0)_m )}{f(\bfX_n\mid G_0 )} \frac{f(\bfX_n\mid G)}{f(\bfX_n\mid (G \cup G_0)_m )}  .
	\eea
	Again by Lemma 2.22 in \cite{lauritzen1996graphical}, there exist a sequence of decomposable graphs $G_0 \subset G_1 \subset \cdots \subset G_k = (G \cup G_0)_m$  with $k = |(G \cup G_0)_m|-|G_0|$, where $G_0, G_1,\ldots, G_{k}$ differ from by exactly one edge.
	For $l=1,\ldots,k$, let $(i_l, j_l)$ be the added edge in the move from $G_{l-1}$ to $G_l$, and $S_l$ is the separator which separates two cliques including $i_l$ and $j_l$ in $G_{l-1}$.
	Similar to $G_0 \subsetneq G$ case, on $\cap_{l=1}^k N_l(C_1)^c$ for any constant $C_1 > 4 + \epsilon'$ and any sufficiently small constant $\epsilon'>0$, 
	\bea
	\frac{f(\bfX_n\mid (G \cup G_0)_m )}{f(\bfX_n\mid G_0 )} 
	&\le& \prod_{l=1}^k \Big\{  \frac{g}{g+1}  \Big(\frac{\nu + n + |S_l|}{\nu + |S_l|- 1/2}  \Big)^{1/2} \big(1 - \hat{\rho}_{i_l j_l \mid S_l}^2 \big)^{-n/2}   \Big\} \\
	&\le& \Big(  \frac{g}{g+1} \Big)^{k} \Big(\frac{\nu +n}{\nu - 1/2} \Big)^{k/2}  \exp \Big\{ \frac{C_1}{2} k \log (n\vee p)  \Big\} ,
	\eea
	where the first inequality follows from Lemma C.1 in \cite{niu2019bayesian}.
	On the other hand, let $G= G_0' \subset G_1' \subset \cdots G'_{k'} = (G\cup G_0)_m$ be a sequence of decomposable graphs with $k' = | (G\cup G_0)_m|-|G|$, where $G_0', \ldots, G'_{k'}$ differ from by exactly one edge.
	For $l=1,\ldots, k'$, let $(i_l' , j_l')$ be the added edge in the move from $G_{l-1}'$ to $G_l'$, and $S_l'$ is the separator which separates two cliques including $i_l'$ and $j_l'$ in $G_{l-1}'$.
	Because $|G|\le R$ and $|G_0|\le R$, we can choose a minimum triangulation of $G\cup G_0$ so that $|S_l'|\le 3R$ for any $l=1,\ldots, k'$.
	For a given constant $C_1'> 0$, let $N_l' (C_1') = \{ \bfX_n : | \hat{\rho}_{i_l' j_l' \mid S_l'} - \rho_{i_l' j_l' \mid S_l' } |^2 > C_1' \log (n\vee p) /n  \}$.	
	Note that, for some constant $C_1' > 12 + \epsilon'$ and  sufficiently small constant $\epsilon'>0$,
	\bea
	&& \bbP_0 ( \cup_{l=1}^k N_l'(C_1') ) \\
	&\le& \sum_{l=1}^k \bbP_0 (N_l'(C_1')) \\
	&\le& \sum_{l=1}^k \frac{21}{(1- |\rho_{i_l' j_l' \mid S_l'} |)^2}  \exp \Big\{ - (n- |S_l'|) \frac{C_1' \log (n\vee p)}{4n}  \Big\} \Big\{  \frac{n}{C_1'(n- |S_l'|)\log (n\vee p)} \Big\}^{1/2} \\
	&\le& 21 (1- \max_{1 \le l \le k}|\rho_{i_l' j_l' \mid S_l'} |)^{-2} \exp \Big[   - \Big\{ \frac{C_1'}{4}\big( 1- \frac{3R}{n}\big) - 2  \Big\} \log (n\vee p) \Big]  \\
	&\le& 21 \exp \Big[   - \Big\{ \frac{C_1'}{4}\big( 1- \frac{3C_{r} }{ \sqrt{n\log (n\vee p)} }\big) - 3  \Big\} \log (n\vee p) \Big] \,\,=\,\,o(1) ,
	\eea
	by Corollary A.1 in \cite{niu2019bayesian}, where the last inequality follows from Condition (A2).	
	Note that there exists at least one true edge in the move from $G$ to $(G\cup G_0)_m$, so let $(i_{l_0}' , j_{l_0}' )$ be a true edge in $G_0$ such that $\rho_{i_{l_0}' j_{l_0}' \mid S_{l_0}' } \neq 0$.
	On the set $\cap_{l=1}^{k'} N_l'(C_1')^c$, we have 
	\bea
	&& \frac{f(\bfX_n\mid G)}{f(\bfX_n\mid (G \cup G_0)_m )}  \\
	&\le& \prod_{l=1}^{k'} \Big\{  \frac{g+1}{g}  \Big( \frac{\nu + |S_l'|}{\nu+n + |S_l'| - 1/2 } \Big)^{1/2}  \big( 1 -\hat{\rho}^2_{i_l' j_l' \mid S_l'}   \big)^{n/2}  \Big\}  \\
	&\le& \Big( \frac{g+1}{g}\Big)^{k'}   \Big( \frac{\nu + 3R}{\nu +n +3R - 1/2 }\Big)^{k'/2}   \Big\{  1 -  \Big( \rho_{i_{l_0}' j_{l_0}' \mid S_{l_0}' }^2 - \frac{C_1' \log (n\vee p)}{n}  \Big) \Big\}^{n/2} \\
	&\le& \Big( \frac{g+1}{g}\Big)^{k'}   \Big( \frac{\nu + 3R}{\nu +n +3R - 1/2 }\Big)^{k'/2}   \exp \Big\{  -\frac{n}{2} \Big(  \rho_{i_{l_0}' j_{l_0}' \mid S_{l_0}' }^2 - \frac{C_1' \log (n\vee p)}{n}  \Big) \Big\}  \\
	&\le& \Big( \frac{g+1}{g}\Big)^{k'}   \Big( \frac{\nu + 3R}{\nu +n + 3R - 1/2 }\Big)^{k'/2} \exp \Big\{  - \Big( \frac{C_\beta R^2 - C_1'}{2}  \Big)  \log (n\vee p)   \Big\}  ,
	\eea
	by Lemma B.1 in \cite{niu2019bayesian}, Conditions (A1) and (A3).
	
	By combining the above results, for any $G_0 \nsubseteq G$, on the set $\{\cap_{l=1}^k N_l(C_1)^c\} \cap \{\cap_{l=1}^{k'} N_l'(C_1')^c\}$, we have 
	\bea
	&& \frac{f(\bfX_n \mid G)}{f(\bfX_n \mid G_0 )} \\
	&\le& \Big( \frac{g+1}{g}\Big)^{|G_0|-|G|} \exp \Big\{ \frac{C_1}{2} \big( |(G\cup G_0)_m| - |G_0| \big) \log (n\vee p) \Big\} \Big( \frac{\nu+n}{\nu -1/2} \Big)^{k/2} \\
	&&\times \,\, \Big(\frac{\nu+3R}{\nu+n +3R-1/2} \Big)^{k'/2} \exp \Big\{  - \Big( \frac{C_\beta R^2 - C_1'}{2}  \Big)  \log (n\vee p)   \Big\}  \\
	&\le& 2 \exp \Big\{ \alpha (|G_0|-|G|) \log (n \vee p)  \Big\} 
	\exp \Big\{ \frac{C_1}{2} \big( |(G\cup G_0)_m| - |G_0| \big) \log (n\vee p) \Big\} \\
	&& \times \,\, n^{-(|G_0|-|G|)/2} \Big(\frac{1 + \nu/n}{\nu - 1/2} \Big)^{k/2}
	\Big(\frac{\nu+3R}{1 + \nu/n +3R/n - 1/(2n)} \Big)^{k'/2} \\
	&& \times \,\, \exp \Big\{  - \Big( \frac{C_\beta R^2 - C_1'}{2}  \Big)  \log (n\vee p)   \Big\}  \\
	&\le& 2 \exp \Big\{ \alpha (|G_0|-|G|) \log (n \vee p)  \Big\} \,
	n^{-(|G_0|-|G|)/2} \\
	&& \times \,\, \exp \Big\{ \frac{C_1 + 1}{2}   |(G\cup G_0)_m|   \log (n\vee p) \Big\}  \exp \Big\{  - \Big( \frac{C_\beta R^2 - C_1'}{2}  \Big)  \log (n\vee p)   \Big\}   \\
	&\le& 2 \exp \Big\{ \Big( \alpha (|G_0|-|G|) + \frac{C_1'}{2}\Big) \log (n \vee p)  \Big\} \,
	n^{-(|G_0|-|G|)/2} \\
	&& \times \,\, \exp \Big\{  - \Big ( \frac{C_\beta}{2} - C_1-1 \Big) R^2 \log (n\vee p)   \Big\} ,
	\eea
	where the last inequality follows from $|(G \cup G_0)_m| \le |G \cup G_0|^2/2 \le |G|^2+|G_0|^2 \le 2 R^2$ by condition (A1).
	The last expression is of order $o(1)$ by choosing a constant $C_1$ arbitrarily close to $4$, because $C_\beta > 10$.
	Thus, we have
	\bea
	\frac{f(\bfX_n\mid G )}{f(\bfX_n\mid G \cup G_0 )}  &\overset{p}{\lra}&  0 
	\eea
	for any $G_0 \nsubseteq G$ as $n\to\infty$, which completes the proof.
\end{proof}

\begin{proof}[Proof of Theorem \ref{thm:post_ratio}]
	Similar to the proof of Theorem \ref{thm:PBF_cons}, we consider two cases: $G_0 \subsetneq G$ and $G_0 \nsubseteq G$.
	Compared to the ratio of marginal likelihoods in Theorem \ref{thm:PBF_cons}, we only need to consider the additional prior ratio term.	
	
	If $G_0 \subsetneq G$, we focus on the event $\cap_{l=1}^k N_l(C_1)^c$ defined in the proof of Theorem \ref{thm:PBF_cons}.
	Then, by the proof of Theorem \ref{thm:PBF_cons}, we have
	\bea
	&&  \frac{\pi(G \mid \bfX_n)}{\pi(G_0 \mid \bfX)}  \\
	&\le& \frac{f(\bfX_n \mid G)}{f(\bfX_n \mid G_0)} 
	\exp \{ C_\tau (|G_0| - |G | )\log (n\vee p) \,  \}
	\binom{p(p-1)/2}{|G|}^{-1}   \binom{p(p-1)/2}{|G_0|} \\
	&\le& \exp \Big\{  - \Big( \alpha + C_\tau -\frac{1}{2}- \frac{C_1}{2} \Big) (|G|-|G_0|)\log (n\vee p)  \Big\}  \\
	&& \times \,\, \prod_{l=1}^k \bigg\{   \binom{p(p-1)/2}{|G_l|}^{-1}   \binom{p(p-1)/2}{|G_{l-1}|}  \bigg\}  \\
	&\le&  \exp \Big\{  - \Big( \alpha + C_\tau -\frac{1}{2}- \frac{C_1}{2} \Big) (|G|-|G_0|)\log (n\vee p)  \Big\}  \\
	&& \times \,\,  \prod_{l=1}^k \bigg\{  \frac{|G_{l-1}| +1}{p(p-1)/2 -|G_{l-1}| } \bigg\} \\
	&\le&  \exp \Big\{  - \Big( \alpha + C_\tau -\frac{1}{2}- \frac{C_1}{2} \Big) (|G|-|G_0|)\log (n\vee p)  + (|G|-|G_0|) \log R \Big\}  \\
	&\le& \exp \Big\{  - \Big( \alpha + C_\tau -1 - \frac{C_1}{2} \Big) (|G|-|G_0|)\log (n\vee p)   \Big\} 
	\eea
	which is of order $o(1)$ by choosing $C_1$ arbitrarily close to $4$, because $\alpha + C_\tau > 3$.

	If $G_0 \nsubseteq G$, we focus on the event $\{\cap_{l=1}^k N_l(C_1)^c\} \cap \{\cap_{l=1}^{k'} N_l'(C_1')^c\}$ defined in the proof of Theorem \ref{thm:PBF_cons}.
	Then, by the proof of Theorem \ref{thm:PBF_cons}, we have
	\bea
	&&  \frac{\pi(G \mid \bfX_n)}{\pi(G_0 \mid \bfX)}   \\
	&=& \frac{f(\bfX_n \mid G)}{f(\bfX_n \mid G_0)} 
	\exp \{ C_\tau (|G_0| - |G | )\log (n\vee p) \,  \}
	\binom{p(p-1)/2}{|G|}^{-1}   \binom{p(p-1)/2}{|G_0|}  \\
	&\le& 2 \exp \Big\{ \Big( \alpha (|G_0|-|G|) + \frac{C_1'}{2}\Big) \log (n \vee p)  \Big\} \,
	n^{-(|G_0|-|G|)/2} \\
	&& \times \,\, \exp \Big\{  - \Big ( \frac{C_\beta}{2} - C_1-1 \Big) R^2 \log (n\vee p)   \Big\} \\
	&& \times \,\, \exp \Big\{  \big(C_\tau |G_0|  + 2|G_0|\big) \log (n\vee p)  \Big\} ,
	\eea
	which is of order $o(1)$ by choosing $C_1$ arbitrarily close to $4$, because $C_\beta > 10$.
\end{proof}

\begin{proof}[Proof of Theorem \ref{thm:selection_cons}]
	 
	Note that
	\bean
	\pi (G \neq G_0 \mid \bfX_n)
	&=& \pi ( G_0 \subsetneq G  \mid \bfX_n)  + \pi (G_0 \nsubseteq G \mid \bfX_n)  \nonumber\\
	&\le& \sum_{G: G_0 \subsetneq G} \frac{\pi ( G  \mid \bfX_n)}{\pi ( G_0  \mid \bfX_n)}  + \sum_{G: G_0 \nsubseteq G} \frac{\pi ( G  \mid \bfX_n)}{\pi ( G_0  \mid \bfX_n)}  . 	\label{selection_twoparts}
	\eean
	For a given constant $C_1>0$, we define 
	\bea
	I_{d}   &=&  \big\{(i,j, S) :    1\le i <j \le p, \,\,  S \subset V \setminus \{i,j\} ,\,\, |S| \le 3R  \\
	&& \quad\quad\quad\quad \,\,  (i,j)\in E_0 \text{ if and only if } \rho_{ij \mid S} = 0  \big\}  ,  \\
	N_{ijS,1}(C_1) 
	&=& \Big\{  \bfX_n: |\hat{\rho}_{ij \mid S} |^2 > \frac{C_1 R \log (n\vee p)}{n}  \Big\}   
	\eea
	for all  $(i,j,S)$ such that $\rho_{ij \mid S} = 0$ and 
	\bea
	N_{ijS,2}(C_1) 
	&=& \Big\{  \bfX_n: |\hat{\rho}_{ij \mid S} - \rho_{ij \mid S}|^2 > \frac{2C_1 R \log (n\vee p)}{n}    \Big\}   
	\eea
	for all  $(i,j,S)$ such that $\rho_{ij \mid S} \neq 0$.
	Let $N_{ijS}(C_1) = N_{ijS,1}(C_1)  \cup N_{ijS,2}(C_1) $.
	Then by Corollary A.1 in \cite{niu2019bayesian}, 
	\bea
	&&  \bbP_0  \Big( \bigcup_{(i,j,S) \in I_d } N_{ijS}(C_1)  \Big) \\
	&\le& \sum_{(i,j,S)\in I_d } \Big\{  \bbP_0 ( N_{ijS,1}(C_1) ) + \bbP_0 ( N_{ijS,2}(C_1) ) \Big\} \\
	&\le& \sum_{(i,j,S)\in I_d } \frac{21}{(1- |\rho_{i j \mid S} |)^2}  \exp \Big\{ - (n- 3R) \frac{C_1 R \log (n\vee p)}{2n}  \Big\} \\
	&& \quad\quad\quad\quad \times \,\,\Big\{  \frac{n}{C_1 R (n- 3R)\log (n\vee p)} \Big\}^{1/2}   \\
	&\le& \sum_{|S|=0}^{3R} \binom{p}{|S|}\binom{p-|S|}{2}  21 (1- \max_{(i,j,S)\in I_d}|\rho_{i j \mid S} |)^{-2} \exp \Big[   - \Big\{ \frac{C_1 R}{2}\big( 1- \frac{3R}{n}\big)  \Big\} \log (n\vee p) \Big]  \\
	&\le& \sum_{s=0}^{3R} p^{s+2}  21 \exp \Big[   - \Big\{ \frac{C_1 R}{2}\big( 1- \frac{3R}{n}\big)  -1 \Big\} \log (n\vee p) \Big]  \\
	&\le& 21 p^{3R +2} \exp \Big[   - \Big\{ \frac{C_1 R}{2}\big( 1- \frac{3R}{n}\big)  -1 \Big\} \log (n\vee p) \Big]  \\
	&\le& 21 \exp \Big[   - \Big\{ \frac{C_1 R}{2}\big( 1- \frac{3R}{n}\big)  - 3R - 3 \Big\} \log (n\vee p) \Big] ,
	\eea
	which is of order $o(1)$ if we take the constant $C_1$ such that $C_1 > 6 + \epsilon'$ for any sufficiently small constant $\epsilon'>0$.
	Therefore, we restrict ourselves to event $\cap_{(i,j,S)\in I_d} N_{ijS}(C_1)^c$ in the rest.
	
	The first term in \eqref{selection_twoparts} is bounded above by
	\bea
	&& \sum_{G: G_0 \subsetneq G} \frac{\pi ( G  \mid \bfX_n)}{\pi ( G_0  \mid \bfX_n)}   \\
	&\le&   \sum_{G: G_0 \subsetneq G}   \frac{\pi(G)}{\pi(G_0)}   \frac{f( \bfX_n\mid G)}{f(\bfX_n\mid G_0)}   \\
	&\lesssim&   \sum_{G: G_0 \subsetneq G}   \frac{\pi(G)}{\pi(G_0)}      \exp \Big\{  - (|G| - |G_{0}|)\Big( \alpha  - \frac{1}{2R} - \frac{C_1}{2} \Big) R\log (n\vee p) \Big\}  \\
	&\le& \sum_{G: G_0 \subsetneq G} \frac{\binom{p(p-1)/2}{|G_{0}|}}{ \binom{p(p-1)/2}{|G|} }      \exp \{ - C_\tau \log p \, (|G| - |G_{0}|)  \}   \\
	&& \quad\quad\quad \times \,\,   \exp \Big\{  - (|G| - |G_{0}|)\Big( \alpha - \frac{1}{2R} - \frac{C_1}{2} \Big) R\log (n\vee p)  \Big\}   \\
	&\le& \sum_{s = |G_0| + 1}^{p(p-1)/2 } \binom{p(p-1)/2 - |G_0| }{s- |G_0|}   \frac{\binom{p(p-1)/2}{|G_0|}}{ \binom{p(p-1)/2}{s} }    \exp \{ -   C_\tau (s - |G_{0}|) \log (n\vee p)  \}   \\
	&& \quad\quad\quad \times  \exp \Big\{  - (s - |G_{0}|)\Big( \alpha  - \frac{1}{2R} - \frac{C_1}{2} \Big) R\log (n\vee p)  \Big\}  \\
	&=& \sum_{s = |G_0| + 1}^{p(p-1)/2 }  \binom{s}{s-|G_0|}   \exp \{ -  C_\tau (s - |G_{0}|) \log (n\vee p)  \}    \\
	&& \quad\quad\quad \times \,\, \exp \Big\{  - (s - |G_{0}|) \Big( \alpha  - \frac{1}{2R} - \frac{C_1}{2} \Big) R \log (n\vee p)  \Big\}  \\
	&\le& \sum_{s= |G_0| + 1}^{p(p-1)/2 }  \exp \Big[  -  \Big\{ R\big( \alpha  - \frac{1}{2R} - \frac{C_1}{2} \big)   + C_\tau - 2 \Big\}  (s-|G_0|)\log (n\vee p)   \Big] \,\, =\,\, o(1)
	\eea
	by taking a constant $C_1$ arbitrarily close to $6$ because $g = (n\vee p)^{- R \alpha}$ and $\alpha > 3$.
	
	Now we focus on the second term in \eqref{selection_twoparts}.
	Note that
	\bea
	&&\sum_{G: G_0 \nsubseteq G} \frac{\pi ( G  \mid \bfX_n)}{\pi ( G_0  \mid \bfX_n)}   \\
	&\le&   \sum_{G: G_0 \nsubseteq G}   \frac{\pi(G)}{\pi(G_0)}   \frac{f( \bfX_n\mid G)}{f(\bfX_n\mid G_0)}   \\
	&\le&  \sum_{G: G_0 \nsubseteq G}   \frac{\binom{p(p-1)/2}{|G_{0}|}}{ \binom{p(p-1)/2}{|G|} }   \exp \big\{  - C_\tau (|G|- |G_0|)\log (n\vee p) \big\}  \\
	&& \times \,\, \frac{f(\bfX_n\mid (G \cup G_0)_m )}{f(\bfX_n\mid G_0 )} \frac{f(\bfX_n\mid G)}{f(\bfX_n\mid (G \cup G_0)_m )}   \\
	&\le&  \sum_{G: G_0 \nsubseteq G}  \frac{\binom{p(p-1)/2}{|G_{0}|}}{ \binom{p(p-1)/2}{|G|} }  \exp \big\{  - C_\tau (|G|- |G_0|)\log (n\vee p) \big\}  \\
	&& \times \,\, n^{-(|G_0|-|G|)/2}  
	\Big(\frac{\nu+3R}{1 + \nu/n +3R/n - 1/(2n)} \Big)^{R^2/2} \\
	&& \times \,\, 2 \exp \Big\{ \alpha (|G_0|-|G|) R \log (n \vee p)  \Big\}  \\
	&& \times \,\, \exp \Big\{ \frac{C_1}{2} R \big( |(G\cup G_0)_m| - |G_0| \big) \log (n\vee p) \Big\}   \\
	&& \times \,\, \exp \Big\{  - \Big( \frac{C_\beta R^3 - 2 C_1 R}{2}  \Big)  \log (n\vee p)   \Big\}  
	\eea 
	and 
	\bea
	&& \sum_{G: G_0 \nsubseteq G}   \frac{\binom{p(p-1)/2}{|G_{0}|}}{ \binom{p(p-1)/2}{|G|} }   \\
	&\le& \sum_{s = 0}^{p(p-1)/2}  \sum_{t=0}^{(|G_0|-1) \wedge s } \binom{|G_0|}{t} \binom{p(p-1)/2 - |G_0|}{s-t} \frac{\binom{p(p-1)/2}{|G_{0}|}}{ \binom{p(p-1)/2}{s} }   \\
	&=& \sum_{s = 0}^{p(p-1)/2}  \sum_{t=0}^{(|G_0|-1) \wedge s } \binom{s}{t} \binom{p(p-1)/2 - s}{|G_0| - t} \\
	&\le& \sum_{s = 0}^{p(p-1)/2}  \sum_{t=0}^{(|G_0|-1) \wedge s }  (p^2 s)^{|G_0|- t} s^{-(|G_0| - s)} \\
	&\le& \sum_{s = 0}^{p(p-1)/2}  \sum_{t=0}^{  (|G_0|-1) \wedge s } \exp \Big\{  4(|G_0|-t) \log (n\vee p)  - (|G_0|-s) \log s    \Big\}		.
	\eea
	Thus, we have
	\bea
	&&  \sum_{G: G_0 \nsubseteq G} \frac{\pi ( G  \mid \bfX_n)}{\pi ( G_0  \mid \bfX_n)}  \\
	&\le& \sum_{s = 0}^{|G_0|-1}  \sum_{t=0}^{  s } 2 \exp \Big[ \big\{ 4(|G_0|-t)   + (C_\tau + \alpha R) (|G_0|-s)   \big\}  \log (n\vee p)  \Big] \\
	&& \times \,\,  \exp \Big\{ - \Big(   \frac{C_\beta - C_1}{2} + \frac{1}{2R} + \frac{C_1}{R^2} \Big)  R^3 \log (n\vee p)   \Big\} \\
	&+& \sum_{s = |G_0|}^{p(p-1)/2}  \sum_{t=0}^{|G_0|-1}   2 \exp \Big[ \big\{ 4(|G_0|-t)   + \big(C_\tau + \alpha R -\frac{3}{2} \big) (|G_0|-s)   \big\}  \log (n\vee p)  \Big]  \\
	&& \times \,\,  \exp \Big\{ - \Big(   \frac{C_\beta - C_1}{2} + \frac{1}{2R} + \frac{C_1}{R^2} \Big)  R^3 \log (n\vee p)   \Big\} ,
	\eea
	which is of order $o(1)$ by taking a constant $C_1$ arbitrarily close to $6$ and  $C_\beta >6$.
	This completes the proof.
\end{proof}

\begin{proof}[Proof of Theorem \ref{thm:post_conv}]
	Let $\epsilon_n = M \tilde{s}_0^2 \sqrt{\log(n\vee p)/n}$.
	Then,
	\bea
	\bbE_0  \Big\{  \pi \big( \|\Omega- \Omega_0\|_1 \ge  \epsilon_n \mid \bfX_n  \big) \Big\}
	&\le& \bbE_0  \Big\{  \pi \big( \|\Omega- \Omega_0\|_1 \ge  \epsilon_n , G= G_0 \mid \bfX_n  \big) \Big\} \\
	&&+ \,\, \bbE_0 \big\{ \pi(G \neq G_0 \mid \bfX_n)  \big\} .
	\eea
	Note that the last term in the right hand side goes to zero as $n\to\infty$ by Theorem \ref{thm:selection_cons}.
	Since 
	\bea
	&& \pi \big( \|\Omega- \Omega_0\|_1 \ge  \epsilon_n , G= G_0 \mid \bfX_n  \big)  \\
	&=& \pi \big( \|\Omega- \Omega_0\|_1 \ge  \epsilon_n  \mid G= G_0, \bfX_n  \big)   \pi( G=G_0 \mid \bfX_n) ,
	\eea
	it suffices to show that 
	\bea
	\pi \big( \|\Omega- \Omega_0\|_1 \ge  \epsilon_n  \mid G= G_0, \bfX_n  \big) 
	&\overset{p}{\lra}& 0
	\eea
	as $n\to\infty$.

	Let $P_{0,1}^{(j)}, \ldots, P_{0, w_j}^{(j)}$ and $S_{0,1}^{(j)},\ldots, S_{0, w_j' }^{(j)}$ be the cliques and separators, respectively, containing the vertex $j$ in $G_0$, selected while maintaining the perfect ordering.
	Note that $w_j \le \tilde{s}_0$ for any $j$, and 
	\bea
	\Omega &=& \sum_{l=1}^{h_0} \{(\sg_{P_{0,l}} )^{-1}\}^0 - \sum_{l=2}^{h_0} \{(\sg_{S_{0,l}} )^{-1}\}^0
	\eea
	for any $\Omega = \sg^{-1}\in P_{G_0}$ (\cite{lauritzen1996graphical}, page 145), where $(A_{P})^0 = (A_{(i,j)}^0) \in \bbR^{p\times p}$ with $A_{(i,j)}^0 = A_{(i,j)}$ for $i,j\in P$ and $A_{(i,j)}^0 = 0$ otherwise for any matrix $A =(A_{(i,j)})$.
	Thus, we have
	\bean
	&& \pi \big( \|\Omega- \Omega_0\|_1 \ge  \epsilon_n  \mid G= G_0, \bfX_n  \big)  \nonumber\\
	&\le& \pi \Big(  \Big\| \sum_{l=1}^{h_0} \{ (\sg_{P_{0,l}})^{-1}  - (\sg_{0,P_{0,l}})^{-1} \}^0  \Big\|_1 \ge \frac{\epsilon_n}{2}    \mid G= G_0 , \bfX_n \Big) \nonumber\\
	&+& \pi \Big(  \Big\| \sum_{l=2}^{h_0} \{ (\sg_{S_{0,l}})^{-1}  - (\sg_{0,S_{0,l}})^{-1} \}^0  \Big\|_1 \ge \frac{\epsilon_n}{2}    \mid G= G_0 , \bfX_n \Big)  \nonumber \\
	&\le&\hspace{-.3cm} \pi \Big(  \max_{1\le j \le p} \Big\| \Big[ \sum_{l=1}^{h_0} \{ (\sg_{P_{0,l}})^{-1}  - (\sg_{0,P_{0,l}})^{-1} \}^0 \Big]_{(\cdot, j)} \Big\|_1 \ge \frac{\epsilon_n}{2}    \mid G= G_0 , \bfX_n \Big) \label{P_term}\\
	&\hspace{-.5cm}+&\hspace{-.5cm}  \pi \Big( \max_{1\le j \le p}  \Big\| \Big[ \sum_{l=2}^{h_0} \{ (\sg_{S_{0,l}})^{-1}  - (\sg_{0,S_{0,l}})^{-1} \}^0  \Big]_{(\cdot, j)}  \Big\|_1 \ge \frac{\epsilon_n}{2}    \mid G= G_0 , \bfX_n \Big)  , \label{S_term}
	\eean 
	where $A_{(\cdot, j)}$ is the $j$ column of $A$  for any matrix $A$.
	Let $\|A\| := \sup_{x\in \bbR^p, \|x\|_2=1 } \| Ax\|_2$ be the spectral norm of a matrix $A$.
	Then, 
	\bea
	 && \max_{1\le j \le p}\Big\| \Big[ \sum_{l=1}^{h_0} \{ (\sg_{P_{0,l}})^{-1}  - (\sg_{0,P_{0,l}})^{-1} \}^0 \Big]_{(\cdot, j)} \Big\|_1  \\
	 &\le& \max_{1\le j \le p} \sum_{l=1}^{w_j}  \big\|  (\sg_{P_{0,l}^{(j)}})^{-1} -   (\sg_{0, P_{0,l}^{(j)}})^{-1}   \big\|_1  \\
	 &\le& \max_{1\le j \le p} \max_{1\le l \le w_j} \tilde{s}_0  \sqrt{|P_{0,l}^{(j)}|}  \big\|  (\sg_{P_{0,l}^{(j)}})^{-1} -   (\sg_{0, P_{0,l}^{(j)}})^{-1}   \big\| . 
	\eea
	Hence, \eqref{P_term} is bounded above by 
	\bea
	p \tilde{s}_0 \cdot\max_{1\le j \le p} \max_{1\le l \le w_j} \pi \Big(  \tilde{s}_0  \sqrt{|P_{0,l}^{(j)}|}  \big\|  (\sg_{P_{0,l}^{(j)}})^{-1} -   (\sg_{0, P_{0,l}^{(j)}})^{-1}   \big\| \ge \frac{\epsilon_n}{2}    \mid G= G_0 , \bfX_n \Big) ,
	\eea
	and similarly, \eqref{S_term} is bounded above by
	\bea
	p \tilde{s}_0 \cdot \max_{1\le j \le p} \max_{1\le l \le w_j'} \pi \Big(  \tilde{s}_0  \sqrt{|S_{0,l}^{(j)}|}  \big\|  (\sg_{S_{0,l}^{(j)}})^{-1} -   (\sg_{0, S_{0,l}^{(j)}})^{-1}   \big\| \ge \frac{\epsilon_n}{2}    \mid G= G_0 , \bfX_n \Big) .
	\eea
	For a given index $j \in [p]$, let
	\bea
	N_{1nj} &:=&  \bigcup_{1\le l \le w_j}  \Big\{ \Omega: \|  (\sg_{P_{0,l}^{(j)}})^{-1} -   (\sg_{0, P_{0,l}^{(j)}})^{-1} \|^2 \ge  \frac{M^2}{9} |{P}_{0,l}^{(j)}| \frac{\log(n\vee p)}{n}    \Big\} , \\
	N_{2nj} &:=&  \bigcup_{1\le l \le w_j'}  \Big\{ \Omega: \|  (\sg_{S_{0,l}^{(j)}})^{-1} -   (\sg_{0, S_{0,l}^{(j)}})^{-1} \|^2 \ge  \frac{M^2}{9} |{S}_{0,l}^{(j)}| \frac{\log(n\vee p)}{n}    \Big\} ,
	\eea
	and $N_{nj} = N_{1nj} \cup N_{2nj}$, 
	then, on the event $\cap_{1\le j \le p} N_{nj}^c$, for example,
	\bea
	\tilde{s}_0 \sqrt{|{P}_{0,l}^{(j)}|} \| (\sg_{P_{0,l}^{(j)}})^{-1} -   (\sg_{0, P_{0,l}^{(j)}})^{-1} \|  
	&\le& \frac{M}{3}  \tilde{s}_0   |{P}_{0,l}^{(j)}| \Big\{ \frac{\log(n\vee p)}{n}  \Big\}^{1/2}   \\
	&\le& \frac{M}{3} \tilde{s}_0^2 \Big\{ \frac{\log(n\vee p)}{n}  \Big\}^{1/2} .
	\eea
	Similar inequalities hold using $S_{0,l}^{(j)}$ instead of $P_{0,l}^{(j)}$.
	Thus, we complete the proof by showing that
	\bea
	\pi \Big(  \bigcup_{j=1}^p N_{1nj}  \mid G=G_0,\bfX_n\Big)  \,\,\le\,\,
	p \tilde{s}_0 \,\max_j \pi \Big(  N_{1nj}  \mid G=G_0,\bfX_n\Big)  &\overset{p}{\lra}& 0
	\eea
	as $n\to\infty$ because $N_{2nj}$ can be dealt with using similar techniques.

	For any $j\in [p]$, 
	\bean
	&& \bbE_0 \Big\{ \pi \Big(  N_{1nj}  \mid G=G_0,\bfX_n\Big)  \Big\}  \nonumber \\
	&\le& \sum_{1\le l \le w_j} \bbE_0 \Big\{ \pi \Big(   \| (\sg_{P_{0,l}^{(j)}})^{-1} -   (\sg_{0, P_{0,l}^{(j)}})^{-1} \|^2 \ge  \nonumber \\
	&& \hspace{5cm} \frac{M^2}{9} |{P}_{0,l}^{(j)}| \frac{\log(n\vee p)}{n}   \,\,\Big|\,\, G=G_0,\bfX_n\Big)  \Big\}  \nonumber \\
	&\le&\hspace{-.3cm}  \sum_{1\le l \le w_j}\bbE_0 \Big\{ \pi \Big(   \| (\sg_{P_{0,l}^{(j)}})^{-1}  - \bbE^\pi((\sg_{P_{0,l}^{(j)}})^{-1} \mid \bfX_n) \|^2 \ge  \nonumber\\
	&& \hspace{4.5cm}  \frac{M^2}{36} |{P}_{0,l}^{(j)}| \frac{\log(n\vee p)}{n}   \,\,\Big|\,\, G=G_0,\bfX_n\Big)  \Big\}  \label{post_conv1}\\
	&&+\,\, \sum_{1\le l \le w_j}\bbP_0 \Big\{   \|\bbE^\pi( (\sg_{P_{0,l}^{(j)}})^{-1} \mid \bfX_n)  - (\sg_{0,P_{0,l}^{(j)}})^{-1}  \|^2  \ge  \nonumber \\
	&& \hspace{4.5cm} \frac{M^2}{36} |{P}_{0,l}^{(j)}| \frac{\log(n\vee p)}{n}   \Big\} , \label{post_conv2}
	\eean
	where $\bbE^\pi( (\sg_{{P}_{0,l}^{(j)} })^{-1}  \mid \bfX_n) $ is the posterior mean of $(\sg_{{P}_{0,l}^{(j)} })^{-1}$.
	By the property of the $G$-Wishart distribution, for any complete subset ${P}_{0,l}^{(j)}$ in $G_0$,  we have
	$ (\sg_{{P}_{0,l}^{(j)} })^{-1} \mid \bfX_n \sim W_{|{P}_{0,l}^{(j)}|}(n+\nu,  (1+g) (\bfX_n^T \bfX_n)_{{P}_{0,l}^{(j)}} )$ (\cite{roverato2002hyper}, Corollary 2).
	Here $W_q(\nu, A)$ denotes the Wishart distribution for $q\times q$ positive definite matrices $B$ with the probability density proportional to $\det(B)^{(\nu-2)/2} \exp \{ -  tr(B A) /2 \}$.
	Thus, we have $\bbE^\pi( (\sg_{{P}_{0,l}^{(j)} })^{-1} \mid \bfX_n)  = (n+\nu +|{P}_{0,l}^{(j)}| -1) (1+g)^{-1} (\bfX_n^T \bfX_n)_{{P}_{0,l}^{(j)}}^{-1}$, where $(\bfX_n^T \bfX_n)_{{P}_{0,l}^{(j)}}^{-1}$ is the inverse of $(\bfX_n^T \bfX_n)_{{P}_{0,l}^{(j)}}$.
	Note that 
	\bea
	\|\bbE^\pi( (\sg_{{P}_{0,l}^{(j)} })^{-1} \mid \bfX_n) \|
	&=& \{ 1+(\nu +|{P}_{0,l}^{(j)}| -1)/n \} (1+g)^{-1} \| (n^{-1} \bfX_n^T \bfX_n)_{{P}_{0,l}^{(j)}  }^{-1} \| \\
	&\le& (2 + \tilde{s}_0/n ) \max_{1\le l \le w_j}\| (n^{-1} \bfX_n^T \bfX_n)_{{P}_{0,l}^{(j)}}^{-1} \|  \\
	&\le& 3 \max_{1\le l \le w_j}\| (n^{-1} \bfX_n^T \bfX_n)_{{P}_{0,l}^{(j)}}^{-1} \| .
	\eea
	For a given constant $C_\lambda >0$, define the set
	\bea
	\tilde{N}_{nj}(C_\lambda) &:=& \Big\{  \bfX_n:   \max_{1\le l \le w_j}\| (n^{-1} \bfX_n^T \bfX_n)_{{P}_{0,l}^{(j)}}^{-1} \| >   C_\lambda/3  \Big\},
	\eea
	then $\|\bbE^\pi( (\sg_{{P}_{0,l}^{(j)} })^{-1} \mid \bfX_n) \| \le  C_\lambda$ on the event $\tilde{N}_{nj}(C_\lambda)^c$.
	By Lemma B.6 in \cite{lee2018optimal}, the posterior probability inside the expectation in \eqref{post_conv1} is bounded above by 
	\bea
	5^{ |{P}_{0,l}^{(j)}|} \big\{  e^{ - c_1 (n+\nu) M^2  |{P}_{0,l}^{(j)}| \log(n\vee p)/n    }    +  e^{ - c_2(n+\nu) M \sqrt{|{P}_{0,l}^{(j)}|\log(n\vee p)/n}  } \big\} 
	\eea
	on the event $\tilde{N}_{nj}(C_\lambda)^c$, for some positive constants $c_1$ and $c_2$ depending only on $C_\lambda$.
	We note here that we are using different parametrization for Wishart and inverse Wishart distributions compared to \cite{lee2018optimal}.
	Moreover, by Lemma B.7 in \cite{lee2018optimal} and Condition (B4),  
	\bea
	&& \bbP_0 \big( \tilde{N}_{nj}(C_\lambda) \big) \\
	&=&  \bbP_0 \Big( \max_{1\le l \le w_j} \| (n^{-1} \bfX_n^T \bfX_n)_{{P}_{0,l}^{(j)}}^{-1} \| >   C_\lambda/3    \Big)  \\
	&\le& \sum_{1\le l \le w_j}  \bbP_0 \Big( \| (n^{-1} \bfX_n^T \bfX_n)_{{P}_{0,l}^{(j)}}^{-1} \| >   C_\lambda/3 \Big)   \\
	&\le& \sum_{1\le l \le w_j}  \bbP_0 \Big( \|\sg_{0,{P}_{0,l}^{(j)} }\| \| (\sg_{0,{P}_{0,l}^{(j)} })^{-1/2}  (n^{-1} \bfX_n^T \bfX_n)_{{P}_{0,l}^{(j)}}^{-1} (\sg_{0,{P}_{0,l}^{(j)} })^{-1/2}\| >   C_\lambda/3 \Big)   \\
	&\le& \sum_{1\le l \le w_j}  \bbP_0 \Big( \epsilon_0^{-1} \| (\sg_{0,{P}_{0,l}^{(j)} })^{-1/2}  (n^{-1} \bfX_n^T \bfX_n)_{{P}_{0,l}^{(j)}}^{-1} (\sg_{0,{P}_{0,l}^{(j)} })^{-1/2}\| >   C_\lambda/3 \Big)   \\
	&=& \sum_{1\le l \le w_j}  \bbP_0 \Big(  \lambda_{\min}( (\sg_{0,{P}_{0,l}^{(j)} })^{-1/2}  (n^{-1} \bfX_n^T \bfX_n)_{{P}_{0,l}^{(j)}}^{-1} (\sg_{0,{P}_{0,l}^{(j)} })^{-1/2} ) <   3/ (\epsilon_0 C_\lambda) \Big)  \\ 
	&\le& \sum_{1\le l \le w_j}  2 e^{ -n (1- \sqrt{|{P}_{0,l}^{(j)}|/n})^2 /8 }  \\
	&\le& 2  p e^{ -n (1- \sqrt{ \tilde{s}_0/n})^2 /8 }    \,\,=\,\,  o( (p\tilde{s}_0)^{-1} ) 
	\eea
	for some large $C_\lambda$ because $\log p =o(n)$ and $( n^{-1} \bfX_n^T \bfX_n)_{{P}_{0,l}^{(j)}} \sim  W_{|{P}_{0,l}^{(j)}|} (n - |{P}_{0,l}^{(j)}|+1 , n (\sg_{0,{P}_{0,l}^{(j)} })^{-1} )$, 
	$$(\sg_{0,{P}_{0,l}^{(j)} })^{-1/2}  (n^{-1} \bfX_n^T \bfX_n)_{{P}_{0,l}^{(j)}}^{-1} (\sg_{0,{P}_{0,l}^{(j)} })^{-1/2} \sim W _{|{P}_{0,l}^{(j)}|} (n - |{P}_{0,l}^{(j)}|+1 ,  n I_{|{P}_{0,l}^{(j)}|} )$$ 
	and $ \tilde{s}_0 = o(n)$.
	Thus, it is easy to show that \eqref{post_conv1} is of order $o((p\tilde{s}_0)^{-1})$.

	Now we focus on \eqref{post_conv2} term to complete the proof. 
	Note that $\bbE^\pi( (\sg_{{P}_{0,l}^{(j)} })^{-1} \mid \bfX_n)  = (n+\nu +|{P}_{0,l}^{(j)}| -1) (1+g)^{-1} (\bfX_n^T \bfX_n)_{{P}_{0,l}^{(j)}}^{-1}$
	and
	$(n^{-1} \bfX_n^T \bfX_n)_{{P}_{0,l}^{(j)} }^{-1}  \sim  IW_{|{P}_{0,l}^{(j)}|} (n - |{P}_{0,l}^{(j)}|+1,  n(\sg_{0, {P}_{0,l}^{(j)}  })^{-1}  ) $.
	Here, $IW_q(\nu, A)$ denotes the inverse Wishart distribution for $q\times q$ positive definite matrices $B$ with the probability density proportional to $\det(B)^{-(\nu +2q)/2} \exp \{ -  tr(B^{-1} A) /2 \}$.
	Also note that \eqref{post_conv2} is bounded above by
	\bean
	 &&\hspace{-.7cm} \sum_{1\le l \le w_j}\bbP_0 \Big\{   \|  (n^{-1} \bfX_n^T \bfX_n)_{{P}_{0,l}^{(j)} }^{-1}   -  (\sg_{0,{P}_{0,l}^{(j)} })^{-1}   \|^2  \ge \frac{M^2}{144} |{P}_{0,l}^{(j)}| \frac{\log(n\vee p)}{n}   \Big\}   \label{post_conv21} \\
	&+&\hspace{-.5cm}    \sum_{1\le l \le w_j}\bbP_0 \Big\{   \|  \frac{  (\nu + |{P}_{0,l}^{(j)} |  -1)/n   - g}{1+g} (n^{-1} \bfX_n^T \bfX_n)_{{P}_{0,l}^{(j)} }^{-1}    \|^2  \ge  \nonumber\\
	&& \hspace{6cm} \frac{M^2}{144} |{P}_{0,l}^{(j)}| \frac{\log(n\vee p)}{n}   \Big\} .  \label{post_conv22}
	\eean
	Note that \eqref{post_conv22} is bounded above by
	\bea
	&& \sum_{1\le l \le w_j}\bbP_0 \Big\{   \|  \frac{  \nu + |{P}_{0,l}^{(j)} |  }{n} (n^{-1} \bfX_n^T \bfX_n)_{{P}_{0,l}^{(j)}  }^{-1}    \|^2  \ge \frac{M^2}{144} |{P}_{0,l}^{(j)}| \frac{\log(n\vee p)}{n}   \Big\} \\
	&\le& \sum_{1\le l \le w_j}\bbP_0 \left\{   \|   (n^{-1} \bfX_n^T \bfX_n)_{{P}_{0,l}^{(j)} }^{-1}    \|  \ge \frac{M}{12} \frac{ \sqrt{n  |{P}_{0,l}^{(j)}| \log(n\vee p)} }{ \nu + |{P}_{0,l}^{(j)} |  } \,  \right\}   \\
	&\le& p \, \bbP_0 \Big\{  \max_{1\le l\le w_j} \|   (n^{-1} \bfX_n^T \bfX_n)_{{P}_{0,l}^{(j)} }^{-1}    \|  \ge  C_\lambda /3 \,  \Big\} \\
	&\le& 2 p^2  e^{ -n (1- \sqrt{ \tilde{s}_0/n})^2 /8 }    \,\,=\,\, o( (p \tilde{s}_0)^{-1}  )
	\eea
	for all sufficiently large $n$ and some constant $C_\lambda>0$, where the last inequality follows from Lemma B.7 in \cite{lee2018optimal}.
	Also note that
	\bea
	&& \|  (n^{-1} \bfX_n^T \bfX_n)_{{P}_{0,l}^{(j)} }^{-1}   - (\sg_{0,{P}_{0,l}^{(j)} })^{-1}  \|  \\
	&\le& \|(n^{-1} \bfX_n^T \bfX_n)_{{P}_{0,l}^{(j)} }^{-1}\|  \cdot \| (\sg_{0,{P}_{0,l}^{(j)} })^{-1} \| \cdot  \|  n^{-1} (\bfX_n^T \bfX_n)_{{P}_{0,l}^{(j)} }   - \sg_{0, {P}_{0,l}^{(j)}  } \|  \\
	&\le& \frac{C_\lambda}{3}  \cdot \epsilon_0 \cdot  \|  n^{-1} (\bfX_n^T \bfX_n)_{{P}_{0,l}^{(j)} }   - \sg_{0, {P}_{0,l}^{(j)}  } \|
	\eea
	on the event $\tilde{N}_{nj}(C_\lambda)^c$, where the last inequality follows from Condition (B4).
	Since 
	\bea
	n^{-1} (\bfX_n^T \bfX_n)_{{P}_{0,l}^{(j)} }  &\sim&  W_{|{P}_{0,l}^{(j)}|} (n - |{P}_{0,l}^{(j)}|+1,  n(\sg_{0, {P}_{0,l}^{(j)}  })^{-1}  ) 
	\eea 
	with $\bbE_0  \{n^{-1} (\bfX_n^T \bfX_n)_{{P}_{0,l}^{(j)} }\} = \sg_{0, {P}_{0,l}^{(j)} }$ and $\|\sg_{0, {P}_{0,l}^{(j)} }\| \le \epsilon_0^{-1}$, the upper bound of \eqref{post_conv21} is given by
	\bea
	&& \hspace{-.5cm}  \sum_{1\le l \le w_j}\bbP_0 \Big\{   \|  n^{-1} (\bfX_n^T \bfX_n)_{{P}_{0,l}^{(j)} }   - \sg_{0, {P}_{0,l}^{(j)}  } \| \ge \frac{M}{4 C_\lambda \epsilon_0} \sqrt{|{P}_{0,l}^{(j)}| \frac{\log(n\vee p)}{n}  } \Big\} \\
	&\le&  \sum_{1\le l \le w_j}  5^{|{P}_{0,l}^{(j)} |}  \Big\{ e^{- c_1 |{P}_{0,l}^{(j)} | \log(n\vee p) }   +  e^{- c_2  \sqrt{n |{P}_{0,l}^{(j)} | \log(n\vee p) }  }  \Big\} 
	\,\,=\,\, o(  (p\tilde{s}_0)^{-1} )
	\eea
	for some constants $c_1$ and $c_2$ depending on $M$ and $\epsilon_0$, by Lemma B.6 in \cite{lee2018optimal}.
	It completes the proof.	
\end{proof}

\begin{proof}[Proof of Theorem \ref{thm:cons_BE}]
	Because
	\bea
	&&  \bbP_0 \bigg(  \, \big\|  \mathbb{E}^\pi (\Omega \mid \widehat{G}, \mathbf{X}_n) - \Omega_0 \big\|_1  \ge M \tilde{s}_0^2 \sqrt{ \frac{\log (n\vee p)}{n}}   \,\, \bigg)   \\
	&\le& \bbP_0 \bigg(  \, \big\|  \mathbb{E}^\pi (\Omega \mid G_0, \mathbf{X}_n) - \Omega_0 \big\|_1  \ge M \tilde{s}_0^2 \sqrt{ \frac{\log (n\vee p)}{n}}   \,\, \bigg) 
	+ \bbP_0 \big( \widehat{G} \neq G_0 \big)
	\eea
	and $\bbP_0 \big( \widehat{G} \neq G_0 \big) \lra 0$ as $n\to\infty$ by Theorem \ref{thm:post_ratio}, 
	it suffices to show that 
	\bea
	\bbP_0 \bigg(  \, \big\|  \mathbb{E}^\pi (\Omega \mid G_0, \mathbf{X}_n) - \Omega_0 \big\|_1  \ge M \tilde{s}_0^2 \sqrt{ \frac{\log (n\vee p)}{n}}   \,\, \bigg) 
	&\lra& 0
	\eea
	as $n\to\infty$.

	By the decomposability of $G_0$ and the posterior mean of $G$-Wishart distribution (\cite{banerjee2014posterior}, page 2119),  we have
	\bea
	&&  \mathbb{E}^\pi (\Omega \mid G_0, \mathbf{X}_n)  \\
	&=& \sum_{l=1}^{h_0}  \frac{n + \nu + |P_{0,l}|-1}{1+g} \big\{ (\bfX_n^T \bfX_n)_{P_{0,l}}^{-1}  \big\}^0  \\
	&& +\,\, \sum_{l=2}^{h_0}  \frac{n + \nu + |S_{0,l}|-1}{1+g} \big\{ (\bfX_n^T \bfX_n)_{S_{0,l}}^{-1}  \big\}^0  \\
	&\equiv& \sum_{l=1}^{h_0}  \big\{  \bbE^\pi \big(  (\sg_{P_{0,l}} )^{-1}  \mid \bfX_n\big)   \big\}^0 + \sum_{l=2}^{h_0}  \big\{ \bbE^\pi \big(  (\sg_{S_{0,l}} )^{-1}  \mid \bfX_n\big)    \big\}^0 .
	\eea 
	Thus,
	\bea
	&& \big\|  \mathbb{E}^\pi (\Omega \mid G_0, \mathbf{X}_n) - \Omega_0 \big\|_1  \\
	&\le& \Big\|  \sum_{l=1}^{h_0} \big\{  \bbE^\pi \big(  (\sg_{P_{0,l}} )^{-1}  \mid \bfX_n\big) - (\sg_{0, P_{0,l}})^{-1}  \big\}^0     \Big\|_1 \\
	&&+\,\, \Big\|  \sum_{l=2}^{h_0} \big\{  \bbE^\pi \big(  (\sg_{S_{0,l}} )^{-1}  \mid \bfX_n\big) - (\sg_{0, S_{0,l}})^{-1}  \big\}^0     \Big\|_1  \\
	&\le& \max_{1\le j \le p} \max_{1\le l \le w_j} \tilde{s}_0 \sqrt{|P_{0,l}^{(j)} |} \, \big\| \bbE^\pi \big(  (\sg_{P_{0,l}} )^{-1}  \mid \bfX_n\big)  -  (\sg_{0, P_{0,l}})^{-1}  \big\|   \\
	&& + \,\, \max_{1\le j \le p} \max_{2\le l \le w_j} \tilde{s}_0 \sqrt{|S_{0,l}^{(j)} |} \, \big\| \bbE^\pi \big(  (\sg_{S_{0,l}} )^{-1}  \mid \bfX_n\big)  -  (\sg_{0, S_{0,l}})^{-1}  \big\|
	\eea
	by the similar arguments used in the proof of Theorem \ref{thm:post_conv}.
	Since we have shown that \eqref{post_conv2} is of order $o( (p\tilde{s}_0)^{-1} )$ in the proof of Theorem \ref{thm:post_conv}, this completes the proof.
\end{proof}

\end{appendix}

\bibliographystyle{dcu}
\bibliography{decomp-graphs}


\end{document}